\documentclass[smallextended]{svjour3}       
\usepackage{graphicx}
\usepackage[utf8]{inputenc}
\usepackage{amssymb,amsmath,array}
\usepackage{eurosym}

\newcommand{\blambda}{{\boldsymbol{\lambda}}}

\newcommand{\G}{\mathcal{G}}

\newcommand{\RR}{\mathbb{R}}
\newcommand{\BB}{\mathbf{B}}

\newcommand{\MM}{\mathbf{P}}
\newcommand{\NN}{\mathbb{N}}

\newcommand{\U}{\mathcal{U}}
\newcommand{\Q}{\mathcal{Q}}

\newcommand{\F}{\mathcal{F}}
\newcommand{\C}{\mathcal{C}}
\newcommand{\CC}{\mathbb{C}}

\newcommand{\minimize}{\mathrm{minimize}\;}
\newcommand{\core}{\mathbf{core}}

\newcommand{\cl}{\mathrm{cl}}
\newcommand{\ext}{\mathrm{ext}}

\newcommand{\QQ}{Q}
\newcommand{\EE}{\mathbb{E}}
\newcommand{\PP}{P}

\renewcommand{\l}{\mathbf{L}}

\renewcommand{\u}{\underline}
\newcommand{\1}{\mathbf{1}}

\newcommand{\upi}{\underline{\pi}}
\newcommand{\opi}{\overline{\pi}}
\newcommand{\ups}{\pi_*}
\newcommand{\ops}{\pi^*}

\makeatletter
\newcommand{\superimpose}[2]{%
  {\ooalign{$#1\@firstoftwo#2$\cr\hfil$#1\@secondoftwo#2$\hfil\cr}}}
\makeatother

\begin{document}

\title{Envelopes of equivalent martingale measures and a generalized no-arbitrage principle in a finite setting}
\subtitle{}

\titlerunning{Envelopes of equivalent martingale measures}        

\author{Andrea Cinfrignini        \and
        Davide Petturiti \and
        Barbara Vantaggi
}


\institute{A. Cinfrignini \at
              Dept. MEMOTEF, La Sapienza University of Rome \\
              \email{andrea.cinfrignini@uniroma1.it}           
\and
           D. Petturiti \at
           Dept. Economics, University of Perugia \\
              \email{davide.petturiti@unipg.it}
\and
           B. Vantaggi \at
           Dept. MEMOTEF, La Sapienza University of Rome \\
              \email{barbara.vantaggi@uniroma1.it}
}

\date{Received: date / Accepted: date}

\maketitle

\begin{abstract}
We consider a one-period market model composed by a risk-free asset and a risky asset with $n$ possible future values (namely, a $n$-nomial market model).
We characterize the lower envelope of the class of equivalent martingale measures in such market model, showing that it is a belief function, obtained as the strict convex combination of two necessity measures. Then, we reformulate a general one-period pricing problem in the framework of belief functions: this allows to model frictions in the market and can be justified in terms of partially resolving uncertainty according to Jaffray. We provide a generalized no-arbitrage condition for a generic one-period market model under partially resolving uncertainty and show that the ``risk-neutral'' belief function arising in the one-period $n$-nomial market model does not satisfy such condition. Finally, we derive a generalized arbitrage-free lower pricing rule through an inner approximation of the ``risk-neutral'' belief function arising in the one-period $n$-nomial market model.

\end{abstract}

\section{Introduction}

As is well-known, the one-period binomial model is the most simple example of financial market showing all features of no-arbitrage pricing (see, e.g., \cite{crr,pliska}). Such model is based on two assets: a risk-free asset (bond) and a risky asset (stock), the latter having only two possible future values. Assuming no-arbitrage condition to hold, which happens under a suitable choice of parameters, this model is proved to be complete, that is there exists a unique equivalent martingale measure that allows to compute no-arbitrage prices as discounted expectations.

As soon as we allow more than two possible values for the risky asset, obtaining a {\it $n$-nomial model}, completeness is lost. In this case, the no-arbitrage condition is equivalent to the existence of an infinite class $\Q$ of equivalent martingale measures (see, e.g., \cite{pliska,cerny}).

One of the basic underlying assumptions of no-arbitrage pricing models is the absence of frictions in the market (see, e.g., \cite{dybvig}), which materializes in the linearity of the price functional. Hence, in case of an incomplete market we need to get rid of non-uniqueness of the equivalent martingale measure for reaching linearity either by completing the market with extra securities or by choosing one of the equivalent martingale measures by some suitable criterion of choice.

Decision models involving sets of probability measures have been extensively studied in the decision theory literature connoting situations of {\it ambiguity} (see, e.g., \cite{tallon-survey,gm}). The incompleteness of the market in a $n$-nomial market model generates a form of ``objective'' ambiguity as one needs to deal with the class of equivalent martingale measures.

In this paper, referring to the one-period $n$-nomial market model, we characterize the corresponding set of equivalent martingale measures for every $n > 2$. We provide a closed form expression for the lower envelope $\underline{Q}$ of the class of equivalent martingale measures, further showing that it is a {\it belief function} in the Dempster-Shafer theory of evidence \cite{dempster,shafer}. In particular, we prove that it can be expressed as the strict convex combination of two {\it necessity measures} \cite{dubprad}. In spite of the very particular form of the lower envelope, the set of equivalent martingale measure $\Q$ is shown not to coincide with $\core(\underline{Q})$ in general, that is with the set of all probability measures dominating $\underline{Q}$ (see, e.g., \cite{grabisch}).

As discussed in \cite{amihud1,amihud2}, real markets show the presence of frictions mainly in the form of bid-ask spreads and this amounts in giving up on the linearity of the price functional. Since we have a set of equivalent martingale measures $\Q$, we could think to use a suitable closed subset $\Q' \subseteq \Q$ to define a {\it lower pricing rule} as a discounted lower expectation, in a way to allow frictions in the market. The approach to pricing through lower/upper expectation functionals has been investigated in several papers (see, e.g., \cite{bensaid,elkaroui,jouni}). The choice of $\Q'$ is not free of issues since a reasonable criterion should be provided. The most natural way to get $\Q'$ is to consider a finite set of random payoffs $\G \subset \RR^\Omega$, and a {\it lower price assessment} $\upi:\G \to \RR$.
The problem is to look for a closed $\Q' \subseteq \Q$ such that
$\upi(X) = \min\limits_{Q \in \Q'} (1+r)^{-1}\EE_Q(X)$, for every $X \in \G$,
where $(1+r)$ is the risk-free return.
Unfortunately, we show that in general this problem could not have a solution.

On the other hand, the fact that $\underline{Q}$ is a belief function suggests to derive a lower pricing rule from it as a discounted Choquet expectation. We stress that, working directly in the framework of belief functions allows to incorporate naturally frictions in the market, nevertheless, for such a lower pricing rule to be acceptable the classical notion of arbitrage must be generalized.

For that, we reformulate a general one-period pricing problem over a finite state space in the framework of belief functions. We provide a generalized avoiding Dutch book condition and a generalized no-arbitrage condition for a lower price assessment
based on the {\it partially resolving uncertainty} principle proposed by  Jaffray in \cite{Jaffray-Bel}.
Adopting such principle, we allow that an agent may only acquire the information that an event $B\neq \emptyset$ occurs, without knowing which is the true state of the world $\omega \in B$. In turn, this translates in considering payoffs of portfolios on every event $B\neq \emptyset$ adopting a systematically pessimistic behavior, that is always considering the minimum of random payoffs. This is in contrast to the usual {\it completely resolving uncertainty} assumption according to which the agent will always acquire which is the true state of the world in $\Omega$.

First, we show that the generalized avoiding Dutch book condition is necessary and sufficient for the existence of a belief function whose corresponding discounted Choquet expectation functional agrees with the lower price assessment, even though the positivity of the belief function cannot be guaranteed.

Second, we prove that the proposed generalized no-arbitrage condition is equivalent to the existence of a strictly positive belief function whose corresponding discounted Choquet expectation functional agrees with the lower price assessment. The theorem we prove is the analog of the {\it first fundamental theorem of asset pricing}, formulated in the context of belief functions. In particular, our result specializes results given in \cite{ckl,cmm}, where the authors characterize an upper pricing rule that can be expressed as a discounted Choquet expectation with respect to a concave (or 2-alternating) capacity. Working with belief functions in place of 2-monotone capacities (that are dual of 2-alternating capacities), allows to introduce non-linearity departing the less from the classical no-arbitrage setting.

Concerning the original problem of deriving a lower pricing rule from the ``risk-neutral'' belief function $\underline{Q}$ arising in the $n$-nomial market model, we show that $\underline{Q}$ does not satisfy the generalized avoiding Dutch book condition, thus we propose a procedure for determining a belief function inner approximating $\underline{Q}$ and giving rise to a generalized arbitrage-free lower pricing rule. Such procedure relies on the choice of a reference equivalent martingale measure $Q_0 \in \Q$, and on the determination of an {\it inner approximating martingale belief function} $\widehat{Bel}$ for $\underline{Q}$ complying only with the stock lower price assessment $\upi(S_1) = S_0$. The latter task is achieved by minimizing a suitable distance, subject to a system of linear constraints, similarly to \cite{mmv,mmv2}.
In this way we get an {\it equivalent inner approximating martingale belief function} $\widehat{Bel}_\epsilon$ for $\underline{Q}$, as the {\it $\epsilon$-contamination} (see, e.g., \cite{huber}) of $Q_0$ with respect to $\core(\widehat{Bel})$. 
Finally, we show that if we further require $\widehat{Bel}$ to comply also with the upper price assessment $\opi(S_1) = S_0$ (arriving to an {\it inner approximating strong martingale belief function}), both $\widehat{Bel}$ and $\widehat{Bel}_\epsilon$ reduce to probability measures.

The paper is structured as follows. In Section~\ref{sec:prelim} we provide some preliminaries. Section~\ref{sec:emm} introduces the one-period market model and provides the characterization of the lower envelope $\underline{Q}$ of the set $\Q$ of equivalent martingale measures. In Section~\ref{sec:gnoarb} we formulate a general one-period pricing problem in the context of belief functions and introduce the generalized avoiding Dutch book condition and the generalized no-arbitrage condition.
Then, in Section~\ref{sec:inner} we cope with the problem of inner approximating the ``risk-neutral'' belief function $\underline{Q}$ arising in the $n$-nomial market model. Finally, Section~\ref{sec:conclusions} collects conclusions and future perspectives.

\section{Preliminaries}
\label{sec:prelim}
Let $\Omega$ be a finite non-empty set and $\F = \mathcal{P}(\Omega)$, where $\mathcal{P}(\Omega)$ denotes the power set of $\Omega$. To avoid cumbersome notation, in the rest of the paper we assume that $\Omega = \{1,\ldots,n\}$ with $n \in \NN$.

\begin{definition}
\label{def:capacity}
A function $\varphi:\F \to [0,1]$ is said a {\bf (normalized) capacity} if:
\begin{itemize}
\item[\it (i)] $\varphi(\emptyset) = 0$ and $\varphi(\Omega) = 1$;
\item[\it (ii)] $\varphi(A) \le \varphi(B)$ when $A \subseteq B$, for every $A,B \in \F$.
\end{itemize}
Further, a capacity $\varphi$ is said a:
\begin{itemize}
\item {\bf probability measure} if $\varphi(A \cup B) = \varphi(A) + \varphi(B)$, for every $A,B \in \F$ with $A \cap B = \emptyset$;
\item {\bf necessity measure} if $\varphi(A \cap B) = \min\{\varphi(A),\varphi(B)\}$, for every $A,B \in \F$;
\item {\bf belief function} if it is {\it completely monotone}, i.e.,
for every $k \ge 2$ and every $A_1,\ldots,A_k \in \F$,
$$
\varphi\left(\bigcup_{i=1}^k A_i\right) \ge \sum_{\emptyset \neq I \subseteq \{1,\ldots,k\}} (-1)^{|I|-1}
\varphi\left(\bigcap_{i \in I} A_i\right);
$$
\item {\bf (coherent) lower probability} if there exists a class $\mathcal{Q}$ of probability measures on $\F$ such that, for every $A \in \F$,
$$
\varphi(A) = \inf_{Q \in \mathcal{Q}} Q(A).
$$
\end{itemize}
\end{definition}

Probability measures and necessity measures are particular belief functions and the later are particular lower probabilities (see, e.g., \cite{grabisch}).

As shown in \cite{walley-lp}, in case of a lower probability we can always consider the closure $\cl(\Q)$ of $\Q$ in the product topology, for which it holds that, for every $A \in \F$,
$$
\varphi(A) = \min_{Q \in \cl(\Q)} Q(A).
$$

Given a capacity $\varphi:\F \to [0,1]$, we can define its {\it dual} capacity $\psi:\F \to [0,1]$ by setting, for every $A \in \F$,
\begin{equation}
\psi(A) = 1 - \varphi(A^c).
\end{equation}
The dual of a necessity measure is said {\it possibility measure},
the dual of a belief function is said {\it plausibility function}, and the dual of a lower probability is said {\it upper probability}.

As proved in \cite{cj-2mon}, every capacity $\varphi$ is completely characterized by its {\it M\"obius inverse} $\mu:\F \to \RR$ through the relations,
\begin{equation}
\mu(A) = \sum_{B \subseteq A} (-1)^{|A \setminus B|} \varphi(B)
\quad \mbox{and}
\quad
\varphi(A) = \sum_{B \subseteq A} \mu(B),
\end{equation}
where $A \in \F$, that imply $\mu(\emptyset) = 0$.

In this paper we will be mainly concerned with belief and plausibility functions that have been introduced by Dempster and Shafer \cite{dempster,shafer} in their {\it theory of evidence}.
Following the usual custom, a belief function is denoted by $Bel$ and its dual plausibility function by $Pl$.

In the rest of the paper $\BB(\Omega,\F)$ stands for the set of all belief functions on $(\Omega,\F)$ and $\MM(\Omega,\F)$ for the subset of all probability measures on $(\Omega,\F)$.

As already pointed out, belief functions are particular lower probabilities as they induce the closed (in the product topology) convex set of probability measures on $\F$ said {\it core}, defined as
\begin{equation}
\core(Bel) = \{Q \in \MM(\Omega,\F) \,:\, Q \ge Bel \},
\end{equation}
that satisfies, for every $A \in \F$,
$$
Bel(A) = \min_{Q \in \core(Bel)} Q(A).
$$
In particular, the M\"obius inverse of a belief function is such that $\mu(A) \ge 0$, for every $A \in \F$, and $\sum_{A \in \F} \mu(A) = 1$. Moreover, the M\"obius inverse of a probability measure can be positive only on singletons, while that of a necessity measure can be positive only on a chain of sets ordered by set inclusion (see, e.g., \cite{grabisch,shafer}).

Denote by $\RR^\Omega$ the set of all real-valued random variables defined on $\Omega$. For $A \in \F$, we denote by $\1_A : \Omega \to \{0,1\}$ the {\it indicator} of $A$, defined as $\1_A(i) = 1$ if $i \in A$ and $0$ otherwise.
In order to avoid cumbersome notation, for every $a \in \RR$, in what follows we identify $a$ with the constant random variable $a\1_\Omega$.

Given a capacity $\varphi$ on $\F$ and $X \in \RR^\Omega$, we can introduce the {\it Choquet expectation} of $X$ with respect to $\varphi$ defined (see, e.g., \cite{denneberg,grabisch}) through the Choquet integral
\begin{equation}
\CC_\varphi(X) = \sum_{i = 1}^n (X(\sigma(i)) - X(\sigma(i+1)))\varphi(E_i^\sigma),
\end{equation}
where $\sigma$ is a permutation of $\Omega$ such that $X(\sigma(1)) \ge \ldots \ge X(\sigma(n))$, $E_i^\sigma = \{\sigma(1),\ldots,\sigma(i)\}$ for $i =1,\ldots,n$, and $X(\sigma(n+1)) = 0$. In particular, if $\varphi$ reduces to a probability measure $Q$, then
$$
\CC_Q(X) = \EE_Q(X),
$$
where $\EE_Q$ denotes the usual expectation operator with respect to $Q$. For a general capacity $\varphi$, $\CC_\varphi$ is not linear, but if $X,Y \in \RR^\Omega$ are {\it comonotonic}, that is $(X(i) - X(j))(Y(i) - Y(j)) \ge 0$ for every $i,j \in \Omega$, then
$$
\CC_\varphi(X + Y) = \CC_\varphi(X) + \CC_\varphi(Y).
$$
Further, $\CC_\varphi$ is always {\it monotonic}, that is, for every $X,Y \in \RR^\Omega$ with $X \le Y$, we have that $\CC_\varphi(X) \le \CC_\varphi(Y)$.

As proved in \cite{schmeidler}, if $\varphi$ reduces to a belief function $Bel$, then
$$
\CC_{Bel}(X) = \min_{Q \in \core(Bel)}\EE_Q(X),
$$
thus $\CC_{Bel}(X)$ can be interpreted as a {\it lower expectation}. Furthermore, in this case it is also possible to provide an expression of $\CC_{Bel}(X)$ relying on the M\"obius inverse of $Bel$. At this aim, let $\U = \F \setminus \{\emptyset\}$. For every random variable $X \in \RR^\Omega$ define the function $X^\l :\U \to \RR$ setting, for every $B \in \U$,
\begin{equation}
\label{eq:lower-gen}
X^\l(B) = \min_{i \in B} X(i).
\end{equation}
If $\mu$ is the M\"obius inverse of $Bel$, then (see, e.g., \cite{gs,grabisch})
\begin{equation}
\CC_{Bel}(X) = \sum_{B \in \U} X^\l(B)\mu(B).
\end{equation}
We further have that $\CC_{Bel}$ is {\it completely monotone} (see, e.g., \cite{deCooman-book}) that is, for every $k \ge 2$ and every $X_1,\ldots,X_k \in \RR^\Omega$, it holds that
$$
\CC_{Bel}\left(\bigvee_{i=1}^k X_i\right)\ge
\sum_{\emptyset \neq I \subseteq \{1,\ldots,k\}}(-1)^{|I|-1}\CC_{Bel}\left(\bigwedge_{i\in I} X_i\right),
$$
where $\vee$ and $\wedge$ denote the pointwise minimum and maximum. In particular, $\CC_{Bel}$ is {\it superadditive}, meaning that, for every $X,Y \in \RR^\Omega$,
$$
\CC_{Bel}(X + Y) \ge \CC_{Bel}(X) + \CC_{Bel}(Y).
$$

Finally, if $Pl$ is the dual plausibility function of $Bel$, then, for every $X \in \RR^\Omega$,
\begin{equation}
\CC_{Pl}(X) = -\CC_{Bel}(-X),
\end{equation}
which turns out to be a {\it completely alternating}, and so {\it subadditive}, Choquet expectation functional (see, e.g., \cite{grabisch,deCooman-book}).

\section{One-period $n$-nomial market model}
\label{sec:emm}
For $n \ge 2$, we consider a one-period {\it $n$-nomial market model} related to times $t = 0$ and $t = 1$, composed by a risky asset (stock) and a risk-free asset (bond). The prices of the two securities are modeled by the processes $\{S_0,S_1\}$ and $\{B_0,B_1\}$ defined on the probability space $(\Omega,\F,\PP)$, with
$\Omega = \{1,\ldots,n\}$, $\PP(\{i\}) = p_i > 0$ for all $i \in \Omega$, and $\F = \mathcal{P}(\Omega)$.

The processes are adapted to the filtration $\{\F_0,\F_1\}$ with $\F_0 = \{\emptyset,\Omega\}$ and $\F_1 = \F$.
Assume $S_0 = s > 0$ and $B_0 = 1$, while the prices at the end of the period satisfy
\begin{equation}
\frac{S_1}{S_0}
=
\left\{
\begin{array}{ll}
m_1, & \mbox{with probability $p_1$},\\
m_2, & \mbox{with probability $p_2$},\\
\vdots & \\
m_n, & \mbox{with probability $p_n$},\\
\end{array}
\right.
\quad
\mbox{and}
\quad
\frac{B_1}{B_0} = 1+r,
\end{equation}
where $m_1 > m_2 > \cdots > m_n > 0$ and $1+r > 0$.

To avoid cumbersome notation, in what follows, every element $\QQ \in \MM(\Omega,\F)$ is identified with the vector $\QQ \equiv (q_1,\ldots,q_n)^T \in [0,1]^n$, where $\QQ(\{i\}) = q_i$, for all $i \in \Omega$.

In this model, the set of {\it equivalent martingale measures} is defined as
\begin{equation}
\Q = \{\QQ \in \MM(\Omega,\F) \;:\; (1+r)^{-1}\EE_{\QQ}(S_1) = S_0, \QQ \sim \PP\}
\end{equation}
where $\QQ \sim \PP$ stands for $\QQ$ is {\it equivalent} to $\PP$.

As is well-known (see, e.g., \cite{pliska}), this set is not  empty if $m_1 > 1+r > m_n$, moreover, $\Q$ is convex but generally not closed.
In the particular case $n = 2$ this set reduces to the singleton $\Q = \{\QQ\}$ where
$$
\QQ \equiv \left(\frac{(1+r) - m_2}{m_1 - m_2},\frac{m_1 - (1+r)}{m_1 - m_2}\right)^T.
$$

To avoid triviality, in what follows we assume $n > 2$, moreover, we denote by $\cl(\Q)$ and $\ext(\cl(\Q))$ the {\it closure} of $\Q$ and the set of {\it extreme points} of the closure. Our aim is to investigate the properties of the lower envelope $\underline{\QQ}$ of $\cl(\Q)$ defined, for every $A \in \F$, as
\begin{equation}
\underline{\QQ}(A) = \min_{\QQ \in\cl(\Q)} \QQ(A).
\end{equation}

We first provide a characterization of $\ext(\cl(\Q))$.

\begin{theorem}
\label{th:extp}
For $n > 2$, if $m_{s-1} > 1+r \geq m_s$ with $s \in \{2,\ldots,n\}$ and $1+r \ne m_n$, let $I = \{1,\ldots,s-1\}$ and $J = \{s,\ldots,n\}$, then
$$
\ext(\cl(\Q)) = \{\QQ_{i,j} \in \MM(\Omega,\F) \;:\; (i,j) \in I \times J\},
$$
where
$\QQ_{i,j} \equiv (0,\ldots,0,q_i,0,\ldots,0,q_j,0,\ldots,0)^T$, with
$$
q_i = \frac{(1+r) - m_j}{m_i - m_j}
\quad
\mbox{and}
\quad
q_j = \frac{m_i - (1+r)}{m_i - m_j}.
$$
\end{theorem}
\begin{proof}
We have that $\QQ\equiv(q_1,\ldots,q_n)^T$ is an element of $\cl(\Q)$ if and only if it solves the system
$$
\left\{
\begin{array}{ll}
\sum_{k = 1}^n q_k = 1,\\[1.5ex]
\sum_{k = 1}^n m_k q_k = 1+r,\\[1ex]
q_k \ge 0, & \mbox{for $k = 1,\ldots,n$}.
\end{array}
\right.
$$
It is immediate to see that the coefficient matrix associated to the first two equations has full rank, so it admits infinite solutions depending on $n - 2$ real parameters.
The set of such solutions is a closed subset of $\RR^n$, while $\cl(\Q)$ is the intersection of such set with the non-negative orthant.

For every $i \in I$ and $j \in J$, $\QQ_{i,j}$ is easily shown to be a solution of the above system.
In particular, if $1+r> m_s$, we have $q_i,q_j \in (0,1)$ for every $i \in I$ and $j \in J$. 

On the other hand, if $1+r= m_s$, we have $q_i,q_j \in (0,1)$ for every $i \in I$ and $j \in J \setminus \{s\}$, while for $j = s$
$$
q_i = \frac{(1+r) - m_s}{m_i - m_s} = \frac{m_s - m_s}{m_i - m_s} =0
\quad
\mbox{and}
\quad
q_j = q_s = \frac{m_i - m_s}{m_i - m_s}=1,
$$
thus $\QQ_{i,j}$ reduces to
$$ \QQ_{i,j}=\QQ_{s} \equiv (0, \dots,0, q_s=1, 0, \dots, 0)^T.$$

The $\QQ_{i,j}$'s are extreme points since they have minimal support and none of them can be written as the convex combination of the others.
\hfill$\square$
\end{proof}

Now we provide a characterization of the lower envelope $\underline{\QQ}$.
\begin{theorem}
\label{th:lenv}
For $n > 2$, if $m_{s-1} > 1+r \geq m_s$ with $s \in \{2,\ldots,n\}$ and $1+r \ne m_n$, let $I = \{1,\ldots,s-1\}$ and $J = \{s,\ldots,n\}$, then, for every $A \in \F$,
$$
\underline{\QQ}(A)
=
\left\{
\begin{array}{ll}
1 & \mbox{if $A = \Omega$},\\[1.5ex]
\frac{(1+r) - m_{\underline{j}}}{m_1 - m_{\underline{j}}} & \begin{array}{l}
 \mbox{if $1+r \ne m_s$ and $I \subseteq A \neq \Omega$}, \\
 \mbox{or $1+r = m_s$ and $ I \cup \{s\}  \subseteq A \neq \Omega$}, \end{array}\\[2.5ex]
\frac{m_{\overline{i}} - (1+r)}{m_{\overline{i}} - m_n} & \mbox{if $J \subseteq A \neq \Omega$},\\[1.5ex]
0 & \mbox{otherwise,}
\end{array}
\right.
$$
where $\underline{j}=\min\{ j\in J\,:\, j\notin A\}$ and $\overline{i}=\max\{ i\in I\,: i \notin A \}$.
\end{theorem}

\begin{proof}
We first prove the case $1+r \neq m_s$.
We have that $\underline{\QQ}(A) = 1$ if and only if for all $(i,j) \in I \times J$, $\{i,j\} \subseteq A$, and this happens if and only if $A = \Omega$. Moreover, $\underline{\QQ}(A) = 0$ if and only if there exists $(i,j) \in I \times J, A \subseteq \{i,j\}^c$, and this happens if and only if $I \not\subseteq A$ and $J \not\subseteq A$.

For the remaining $A$'s, two situations can occur: either {\it (a)} $I \subseteq A\neq \Omega$ or {\it (b)} $J \subseteq A\neq \Omega$.

{\it (a)}. If $I \subseteq A\neq \Omega$, then
\begin{eqnarray*}
\underline{\QQ}(A) &=& \min_{\substack{(i,j) \in I \times J\\ \{i,j\} \subseteq A}} \left[\1_A(i) \frac{(1+r) - m_j}{m_i - m_j} + \1_A(j) \frac{m_i - (1+r)}{m_i - m_j}\right]\\
&=&
\min_{\substack{(i,j) \in I \times J\\ i \in A,j \notin A}} \frac{(1+r) - m_j}{m_i - m_j}.
\end{eqnarray*}

Suppose $i \in I$ and let $j \in J$ be such that $j \notin A$, with $m_1 > m_i > 1+r > m_j$. Since
$$
\frac{(1+r) - m_j}{m_1 - m_j} -
\frac{(1+r) - m_j}{m_i - m_j}
= \frac{((1+r) - m_j)(m_i - m_1)}{(m_1 - m_j)(m_i - m_j)}<0
$$
we have
$\frac{(1+r) - m_j}{m_1 - m_j} <
\frac{(1+r) - m_j}{m_i - m_j}$.
Suppose $j,j' \in J$ are such that $j,j' \notin A$ with $m_1 > 1+r > m_j > m_{j'}$.
Since
$$
\frac{(1+r) - m_j}{m_1 - m_j} -
\frac{(1+r) - m_{j'}}{m_1 - m_{j'}}
= \frac{((1+r) - m_1)(m_j - m_{j'})}{(m_1 - m_j)(m_1 - m_{j'})}<0
$$
we have $\frac{(1+r) - m_j}{m_1 - m_j} <
\frac{(1+r) - m_{j'}}{m_1 - m_{j'}}$.
Hence, if $\underline{j}$ is the minimum element of $J$ such that $\underline{j} \notin A$ we have that
$$
\underline{\QQ}(A) =
\min_{\substack{(i,j) \in I \times J\\ i \in A,j \notin A}} \frac{(1+r) - m_j}{m_i - m_j} = \frac{(1+r) - m_{\underline{j}}}{m_1 - m_{\underline{j}}}.
$$


{\it (b)}. If $J \subseteq A\neq \Omega$, then
\begin{eqnarray*}
\underline{\QQ}(A) &=& \min_{\substack{(i,j) \in I \times J\\ \{i,j\} \subseteq A}} \left[\1_A(i) \frac{(1+r) - m_j}{m_i - m_j} + \1_A(j) \frac{m_i - (1+r)}{m_i - m_j}\right]\\
&=&
\min_{\substack{(i,j) \in I \times J\\ i \notin A,j \in A}} \frac{m_i - (1+r)}{m_i - m_j}.
\end{eqnarray*}

Suppose $j \in J$ and let $i \in I$ be such that $i \notin A$ with $m_i > 1+r > m_j > m_n$. Since
$$
\frac{m_i - (1+r)}{m_i - m_j} -
\frac{m_i -(1+r)}{m_i - m_n}
= \frac{(m_i - (1+r))(m_j - m_n)}{(m_i - m_j)(m_i - m_n)}>0
$$
we have
$\frac{m_i - (1+r)}{m_i - m_j} >
\frac{m_i -(1+r)}{m_i - m_n}$. Suppose $i,i' \in I$ are such that $i,i' \notin A$ with $m_i > m_{i'} > 1+r > m_n$.
Since
$$
\frac{m_i - (1+r)}{m_i - m_n} -
\frac{m_{i'} -(1+r)}{m_{i'} - m_n}
=
\frac{((1+r) - m_n)(m_i - m_{i'})}{(m_i - m_n)(m_{i'} - m_n)}>0
$$
we have
$\frac{m_i - (1+r)}{m_i - m_n} >
\frac{m_{i'} -(1+r)}{m_{i'} - m_n}$.
Hence, if $\overline{i}$ is the maximum element of $I$ such that $\overline{i} \notin A$ we have that
$$
\underline{\QQ}(A) =
\min_{\substack{(i,j) \in I \times J\\ i \notin A,j \in A}} \frac{m_i - (1+r)}{m_i - m_j} =
\frac{m_{\overline{i}} - (1+r)}{m_{\overline{i}} - m_n}.
$$

Finally, we prove the case $1+r = m_s$. As before,
we have that $\underline{\QQ}(A) = 1$ if and only if $A = \Omega$. Moreover, $\underline{\QQ}(A) = 0$ if and only if $I \cup \{s\} \not\subseteq A$ and $J \not\subseteq A$.

For the remaining $A$'s, two situations can occur: either {\it (a')} $I \cup \{s\} \subseteq A\neq \Omega$ or {\it (b')} $J \subseteq A\neq \Omega$. Situation {\it (b')} coincides with {\it (b)}, thus it is proved in the same way.

{\it (a')} If $I \cup \{s\} \subseteq A \neq \Omega$, then proceeding as in the proof of {\it (a)} we have
\begin{eqnarray*}
\underline{\QQ}(A) &=& \min_{\substack{(i,j) \in (I\cup\{s\}) \times J\\ \{i,j\} \subseteq A}} \left[\1_A(i) \frac{(1+r) - m_j}{m_i - m_j} + \1_A(j) \frac{m_i - (1+r)}{m_i - m_j}\right]\\
&=&
\min_{\substack{(i,j) \in (I\cup\{s\}) \times J\\ i \in A,j \notin A}} \frac{(1+r) - m_j}{m_i - m_j} = \frac{(1+r) - m_{\underline{j}}}{m_1 - m_{\underline{j}}}.
\end{eqnarray*}
\hfill$\square$
\end{proof}

In the next theorem we characterize the M\"obius inverse of $\underline{\QQ}$.

\begin{theorem}
\label{th:mobius}
For $n > 2$, if $m_{s-1} > 1+r \geq m_s$ with $s \in \{2,\ldots,n\}$ and $1+r \ne m_n$, let $I = \{1,\ldots,s-1\}$ and $J = \{s,\ldots,n\}$.
Let $\mu:\F \to \RR$ be the M\"obius inverse of $\underline{\QQ}$. Then, for every $A \in \F$,
$$
\mu(A) =
\left\{
\begin{array}{ll}
\frac{(1+r) - m_s}{m_1 - m_s} & \mbox{if $1+r \ne m_s$ and $A = I$},\\[1.5ex]
\frac{(1+r) - m_{k+1}}{m_1 - m_{k+1}} - \frac{(1+r) - m_{k}}{m_1 - m_{k}}  & \mbox{if $A = \{1,\ldots,k\}$ and $I \subset A \neq \Omega$,}\\[1.5ex]
\frac{m_{s-1} - (1+r)}{m_{s-1} - m_n} & \mbox{if $A = J$},\\[1.5ex]
\frac{m_{k-1} - (1+r)}{m_{k-1} - m_n} - \frac{m_k - (1+r)}{m_k - m_n}  & \mbox{if $A = \{k,\ldots,n\}$ and $J \subset A \neq \Omega$},\\[1.5ex]
0 & \mbox{otherwise}.
\end{array}
\right.
$$

\end{theorem}
\begin{proof}
We first prove the case $1+r \neq m_s$, by considering all the possibilities for $A \in \F$.

{\it (a)}.
If $I \not\subseteq A$ and $J \not\subseteq A$, then $\mu(A) = 0$. Indeed, by Theorem~\ref{th:lenv} we have that $\underline{\QQ}(B) = 0$ for every $B \subseteq A$ and this implies $\mu(B) = 0$ for every $B \subseteq A$.

{\it (b)}. If $A=I$, then by Theorem~\ref{th:lenv} the only $B \subseteq A$ with $\underline{Q}(B) \ne 0$ is $ B=A=I$. Hence, $$ \mu(I) = \underline{\QQ}(I) = \frac{(1+r)-m_{s}}{m_1-m_s},$$ as $s$ is the minimum element of $J$ not in $A=I$.

{\it (c)}. If $A=J$, then by Theorem~\ref{th:lenv} the only $ B \subseteq A$ with $\underline{Q}(B) \ne 0$ is $ B=A=J$. Hence, $$ \mu(J) = \underline{\QQ}(J) = \frac{m_{s-1}-(1+r)}{m_{s-1}-m_n},$$ as $s-1$ is the maximum element of $I$ not in $A=J$.

{\it (d)}. If $A \neq \{1,\ldots,k\}$ and $I \subset A \neq \Omega$, then $\mu(A) = 0$.
To see this, let $A =\{1,\ldots,k\} \cup B$ with $I \subseteq \{1,\ldots,k\} \neq \Omega$, $B \neq \emptyset$, and $B \cap \{1,\ldots,k+1\} = \emptyset$. Since for all $E \subseteq A$ not containing $I$ we have $\underline{\QQ}(E) = 0$, we can write
$$
\mu(A) = \sum_{I \subseteq E \subseteq A}(-1)^{|A \setminus E|}\underline{\QQ}(E).
$$
For every $s-1 \le t \le k$, if $E$ contains $\{1,\ldots,t\}$ but not $\{1,\ldots,t+1\}$, it follows that 
$\underline{\QQ}(E) = \frac{(1+r) - m_{t+1}}{m_1 - m_{t+1}}$ and all of such sets are of the form $\{1,\ldots,t\} \cup C$ with $C \subseteq F$, where $F = \{t+2,\ldots,k\} \cup B$ if $t+2 \le k$ and $F = B$ otherwise. Moreover, we have
$$
\begin{array}{l}
\underline{\QQ}(\{1,\ldots,t\} \cup F) 
- \sum_{\substack{D \subseteq F\\|D| = |F| - 1}} \underline{\QQ}(\{1,\ldots,t\} \cup D)\\
\quad + \sum_{\substack{D \subseteq F\\|D| = |F| - 2}} \underline{\QQ}(\{1,\ldots,t\} \cup D)
+ \cdots
+ (-1)^{|F|}\underline{\QQ}(\{1,\ldots,t\}) = 0,
\end{array}
$$
since all terms are equal in absolute value and the number of positive terms is equal to that of negative terms.
In turn, this implies that $\mu(A) = 0$.

{\it (e)}. If $A \neq \{k,\ldots,n\}$ and $J \subset A \neq \Omega$, then $\mu(A) = 0$. The proof of this claim is analogous to point {\it (d)}.

{\it (f)}. If $A=\{1,\dots, k\}$ and $I \subset A \neq \Omega$, i.e., $s \leq k \leq n-1$, then, taking into account points {\it (a)}--{\it (e)},
$$
\underline{\QQ}(A) = \sum_{t = s-1}^k \mu(\{1,\ldots,t\}) = \frac{(1+r) - m_{k+1}}{m_1 - m_{k+1}}.
$$
Hence, we have that
$$
\mu(A) = \underline{\QQ}(A) - \sum_{t = s-1}^{k-1} \mu(\{1,\ldots,t\}) = \frac{(1+r) - m_{k+1}}{m_1 - m_{k+1}} - \frac{(1+r) - m_k}{m_1 - m_k}.
$$

{\it (g)}. If $A=\{k, \dots, n\}$ and $J \subset A \neq \Omega$, i.e., with $2 \leq k \leq s-1$, then proceeding as in point {\it (f)} we get
$\mu(A) = \frac{m_{k-1} - (1+r)}{m_{k-1} - m_n} - \frac{m_k - (1+r)}{m_k - m_n}$.

{\it (h)}. If $A = \Omega$, then $\mu(\Omega) = 0$. Indeed, by points {\it (a)}--{\it (g)}, for every $A \in \mathcal{F}\setminus \{ \Omega \}, \mu(A) \geq 0$ and, in particular, $\mu$ is strictly positive on the families
\begin{eqnarray*}
 \mathcal{C}_1 &=& \left\{ \{1, \dots , s-1 \}, \{ 1, \dots, s\}, \dots , \{ 1, \dots, n-1 \} \right\}, \\
 \mathcal{C}_2 &=& \left\{ \{s, \dots , n \}, \{ s-1, \dots, n\}, \dots , \{ 2, \dots, n \} \right\},
 \end{eqnarray*}
while it is $0$ otherwise. By the properties of the M\"{o}bius inverse, it must be $\sum_{A \in \mathcal{F}} \mu(A) =1$, and since
$$ \sum_{A \in \mathcal{F}\setminus \{\Omega\}} {\mu (A)} = \sum_{A \in \mathcal{C}_1} {\mu(A)} + \sum_{A \in \mathcal{C}_2} {\mu(A)}  = \frac{(1+r)-m_n}{m_1-m_n} - \frac{m_1-(1+r)}{m_1-m_n} =1,$$
it follows that $\mu(\Omega)=0$.

Finally, we prove the case $1+r = m_s$. Proceeding as in points {\it (a)}--{\it (g)} by taking $I \cup \{s\}$ in place of $I$, it is possible to show that        $\mu$ is strictly positive on the families $\mathcal{C}_2$ and
$$ \mathcal{C}'_1 = \{ \{1,\ldots, s\}, \ldots, \{1,\ldots, n-1\}\},$$ while it is $0$ on $\F \setminus (\{\Omega\} \cup \C_1' \cup \C_2)$. Thus, in analogy to point {\it (h)}, since
 $$ \sum_{A \in \mathcal{F}\setminus\{\Omega\}} \mu(A) = \sum_{A \in \mathcal{C}'_1} \mu(A) + \sum_{A \in \mathcal{C}_2} \mu(A) = 1,$$ it follows that $\mu(\Omega) = 0$.
\hfill$\square$
\end{proof}

The previous theorem implies that $\underline{\QQ}$ is completely monotone (i.e., a belief function) and furthermore it can be expressed as the strict convex combination of two necessity measures defined on $\F$.


\begin{corollary}
\label{cor:lower}
The lower probability $\underline{\QQ}$ satisfies the following properties:
\begin{itemize}
\item[\it (i)] $\underline{\QQ}$ is completely monotone, that is, for every $k \ge 2$ and every $A_1,\ldots,A_k \in \F$, it holds that
$$
\underline{\QQ}\left(\bigcup_{i = 1}^k A_i \right)
\ge
\sum_{\emptyset \neq I \subseteq \{1,\ldots,k\}} (-1)^{|I| + 1} \underline{\QQ}\left(\bigcap_{i \in I} A_i\right).
$$
\item[\it (ii)] there exist two necessity measures $N_1,N_2:\F \to [0,1]$ and $\alpha \in (0,1)$ such that, for every $A \in \F$,
$$\underline{\QQ}(A) = \alpha N_1(A) + (1-\alpha)N_2(A).$$
\end{itemize}
\end{corollary}
\begin{proof}
Statement {\it (i)} is an immediate consequence of Theorem~\ref{th:mobius} since $\mu(A) \ge 0$, for every $A \in \F$ (see \cite{cj-2mon,grabisch}). For statement {\it (ii)}, we prove only the case $1+r \neq m_s$ as the other case can be proved similarly. By Theorem~\ref{th:mobius}, the focal elements of $\mu$ form two chains ordered by set inclusion:
\begin{eqnarray*}
\C_1 &=& \{\{1,\ldots,s-1\}, \{1,\ldots,s\}, \ldots, \{1,\ldots,n-1\}\},\\
\C_2 &=& \{\{s,\ldots,n\}, \{s-1,\ldots,n\}, \ldots, \{2,\ldots,n\}\}.
\end{eqnarray*}
Let $\alpha = \sum_{A \in \C_1} \mu(A) = \frac{(1+r) - m_n}{m_1 - m_n}$ and $(1-\alpha) = \sum_{A \in \C_2} \mu(A) = \frac{m_1 - (1+r)}{m_1 - m_n}$ and define $\mu_1,\mu_2$ on $\F$ setting, for every $A \in \F$,
$$
\mu_1(A) =
\left\{
\begin{array}{ll}
\frac{\mu(A)}{\alpha} & \mbox{if $A \in \C_1$},\\[1.5ex]
0 & \mbox{otherwise},
\end{array}
\right.
\quad
\mbox{and}
\quad
\mu_2(A) =
\left\{
\begin{array}{ll}
\frac{\mu(A)}{1-\alpha} & \mbox{if $A \in \C_2$},\\[1.5ex]
0 & \mbox{otherwise}.
\end{array}
\right.
$$
A simple verification shows that $\mu_1,\mu_2$ are non-negative M\"obius inverses with nested focal elements, so they induce, respectively, two necessity measures $N_1,N_2$ on $\F$ (see \cite{grabisch}). Then, by construction we have that, for every $A \in \F$,
$$\underline{\QQ}(A) = \alpha N_1(A) + (1-\alpha)N_2(A).$$
\hfill$\square$
\end{proof}

The following example shows the representation of the lower envelope $\underline{Q}$ as a strict convex combination of two necessity measures.

\begin{example}
\label{ex:lp}
Let $\Omega = \{1,2,3,4\}$ and $m_1 = 4$, $m_2 = 2$, $m_3 = \frac{1}{2}$, $m_4 = \frac{1}{4}$, and $1+r = 1$. To avoid cumbersome notation, we denote events omitting braces and commas. In this case we have $I = \{1,2\}$, $J = \{3,4\}$ and $\ext(\cl(\Q)) = \{\QQ_{1,3}, \QQ_{1,4}, \QQ_{2,3},\QQ_{2,4}\}$ inducing the $\underline{\QQ}$ reported below

\begin{center}
\begin{tabular}{c|cccccccccccccccc}
$\F$ & $\emptyset$ & $1$ & $2$ & $3$ & $4$ & $12$ & $13$ & $14$ & $23$ & $24$ & $34$ & $123$ & $124$ & $134$ & $234$ & $\Omega$\\
\hline
\\[-1.5ex]
$\QQ_{1,3}$
& $0$ & $\frac{15}{105}$ & $0$ & $\frac{90}{105}$ & $0$
& $\frac{15}{105}$ & $1$ & $\frac{15}{105}$ & $\frac{90}{105}$ & $0$ & $\frac{90}{105}$ &
$1$ & $\frac{15}{105}$ & $1$ & $\frac{90}{105}$ & $1$
\\[1ex]
$\QQ_{1,4}$
& $0$ & $\frac{21}{105}$ & $0$ & $0$ & $\frac{84}{105}$
& $\frac{21}{105}$ & $\frac{21}{105}$ & $1$ & $0$ & $\frac{84}{105}$ & $\frac{84}{105}$ &
$\frac{21}{105}$ & $1$ & $1$ & $\frac{84}{105}$ & $1$
\\[1ex]
$\QQ_{2,3}$
& $0$ & $0$ & $\frac{35}{105}$ & $\frac{70}{105}$ & $0$
& $\frac{35}{105}$ & $\frac{70}{105}$ & $0$ & $1$ & $\frac{35}{105}$ & $\frac{70}{105}$ &
$1$ & $\frac{35}{105}$ & $\frac{70}{105}$ & $1$ & $1$
\\[1ex]
$\QQ_{2,4}$
& $0$ & $0$ & $\frac{45}{105}$ & $0$ & $\frac{60}{105}$
& $\frac{45}{105}$ & $0$ & $\frac{60}{105}$ & $\frac{45}{105}$ & $1$ & $\frac{60}{105}$ &
$\frac{45}{105}$ & $1$ & $\frac{60}{105}$ & $1$ & $1$\\[1ex]
\hline
\\[-1.5ex]
$\underline{\QQ}$
& $0$ & $0$ & $0$ & $0$ & $0$
& $\frac{15}{105}$ & $0$ & $0$ & $0$ & $0$ & $\frac{60}{105}$ &
$\frac{21}{105}$ & $\frac{15}{105}$ & $\frac{60}{105}$ & $\frac{84}{105}$ & $1$\\[1ex]
$\mu$
& $0$ & $0$ & $0$ & $0$ & $0$
& $\frac{15}{105}$ & $0$ & $0$ & $0$ & $0$ & $\frac{60}{105}$ &
$\frac{6}{105}$ & $0$ & $0$ & $\frac{24}{105}$ & $0$
\end{tabular}
\end{center}

We have that $\C_1 = \{12,123\}$ and $\C_2 = \{34,234\}$, with
$$
\alpha = \mu(12) + \mu(123) = \frac{21}{105}
\quad
\mbox{and}
\quad
1-\alpha = \mu(34) + \mu(234) = \frac{84}{105}.
$$
The M\"obius inverses $\mu_1,\mu_2$ and the corresponding necessity measures $N_1,N_2$ are defined below
\begin{center}
\begin{tabular}{c|cccccccccccccccc}
$\F$ & $\emptyset$ & $1$ & $2$ & $3$ & $4$ & $12$ & $13$ & $14$ & $23$ & $24$ & $34$ & $123$ & $124$ & $134$ & $234$ & $\Omega$\\
\hline
\\[-1.5ex]
$\mu_1$ & $0$ & $0$ & $0$ & $0$ & $0$ & $\frac{15}{21}$
& $0$ & $0$ & $0$ & $0$ & $0$ & $\frac{6}{21}$ & $0$ & $0$ & $0$ & $0$
\\[1ex]
$N_1$ & $0$ & $0$ & $0$ & $0$ & $0$ & $\frac{15}{21}$
& $0$ & $0$ & $0$ & $0$ & $0$ & $1$ & $\frac{15}{21}$ & $0$ & $0$ & $1$
\\[1ex]
$\mu_2$ & $0$ & $0$ & $0$ & $0$ & $0$ & $0$
& $0$ & $0$ & $0$ & $0$ & $\frac{60}{84}$ & $0$ & $0$ & $0$ & $\frac{24}{84}$ & $0$
\\[1ex]
$N_2$ & $0$ & $0$ & $0$ & $0$ & $0$ & $0$
& $0$ & $0$ & $0$ & $0$ & $\frac{60}{84}$ & $0$ & $0$ & $\frac{60}{84}$ & $1$ & $1$
\\[1ex]
\hline
\\[-1.5ex]
$\alpha N_1 + (1-\alpha)N_2$
& $0$ & $0$ & $0$ & $0$ & $0$
& $\frac{15}{105}$ & $0$ & $0$ & $0$ & $0$ & $\frac{60}{105}$ &
$\frac{21}{105}$ & $\frac{15}{105}$ & $\frac{60}{105}$ & $\frac{84}{105}$ & $1$
\end{tabular}
\end{center}
\hfill$\blacklozenge$
\end{example}

Even though $\underline{Q}$ is a belief function, we have that, for $n > 2$, $\cl(\Q) \neq \core(\underline{Q})$ in general, as shown in the following example.

\begin{example}
Consider $\Omega = \{1,2,3,4\}$, $m_1 = 4$, $m_2 = 2$, $m_3 = \frac{1}{2}$, $m_4 = \frac{1}{4}$, $1+r = 1$, $\Q$ and $\underline{Q}$ of Example~\ref{ex:lp}.

A straightforward computation shows that $\cl(\Q) \neq \core(\underline{Q})$, since (see, e.g., \cite{schmeidler2,grabisch}), assuming $S_0 = s > 0$,
$$
\CC_{\underline{Q}}\left(\frac{S_1}{S_0}\right) = \frac{54}{105} < 1 = \min_{Q \in \cl(\Q)} \EE_Q\left(\frac{S_1}{S_0}\right).
$$
In particular, taking the permutation $\sigma = \langle 1,2,3,4 \rangle$ and defining the probability measure
\begin{eqnarray*}
\QQ^\sigma &\equiv& \left(\underline{Q}(1), \underline{Q}(12) - \underline{Q}(1), \underline{Q}(123) - \underline{Q}(12), \underline{Q}(1234) - \underline{Q}(123)\right)^T\\
&\equiv& \left(0,\frac{15}{105},\frac{6}{105},\frac{84}{105}\right)^T
\end{eqnarray*}
we have that $\QQ^\sigma\notin \cl(\Q)$ which further proves that
$\cl(\Q) \neq \core(\underline{Q})$ (see, e.g., \cite{grabisch}).
\hfill$\blacklozenge$
\end{example}

As is well-known (see, e.g., \cite{delbschac}), given a random variable $X \in \RR^\Omega$ expressing the payoff at time $t = 1$ of a contract, its no-arbitrage price at time $t=0$ can be computed relying on the set of equivalent martingale measures $\Q$, by computing
\begin{equation}
\ups(X) = \min_{Q \in \cl(\Q)} (1+r)^{-1}\EE_Q(X)
\quad \mbox{and} \quad
\ops(X) = \max_{Q \in \cl(\Q)} (1+r)^{-1}\EE_Q(X).
\end{equation}
It holds that (see, e.g., \cite{pliska,cerny}):
\begin{itemize}
\item if $\ups(X) = \ops(X)$, then their common value $\pi(X)$ is the no-arbitrage price at time $t=0$ of $X$;
\item if $\ups(X) < \ops(X)$, then the no-arbitrage price $\pi(X)$ at time $t=0$ of $X$ belongs to the open interval $(\ups(X),\ops(X))$.
\end{itemize}

If we have two contracts with payoffs $X,Y \in \RR^\Omega$ at time $t=1$, then their no-arbitrage price intervals are
$$
(\ups(X),\ops(X))
\quad \mbox{and} \quad
(\ups(Y),\ops(Y)),
$$
nevertheless, we are not free to choose a value in one interval independently of the other, as shown in the following example.

\begin{example}
\label{ex:assessment}
Let $\Omega = \{1,2,3\}$, $m_1 = 4$, $m_2 = 2$,
$m_3 = \frac{1}{2}$, $1+r = 1$ and $S_0 = 20$. In this case we have $I = \{1,2\}$, $J = \{3\}$ and $\ext(\cl(\Q)) = \{\QQ_{1,3}, \QQ_{2,3}\}$ inducing the $\underline{\QQ}$ reported below
\begin{center}
\begin{tabular}{c|cccccccc}
$\F$ & $\emptyset$ & $1$ & $2$ & $3$ & $12$ & $13$ & $23$ & $\Omega$\\
\hline
\\[-1.5ex]
$Q_{1,3}$
& $0$ & $\frac{21}{105}$ & $0$ & $\frac{84}{105}$ & $\frac{21}{105}$
& $1$ & $\frac{84}{105}$ & $1$
\\[1ex]
$Q_{2,3}$
& $0$ & $0$ & $\frac{45}{105}$ & $\frac{60}{105}$ & $\frac{45}{105}$
& $\frac{60}{105}$ & $1$ & $1$\\[1ex]
\hline
\\[-1.5ex]
$\underline{Q}$
& $0$ & $0$ & $0$ & $\frac{60}{105}$ & $\frac{21}{105}$
& $\frac{60}{105}$ & $\frac{84}{105}$ & $1$
\\[1ex]
$\mu$
& $0$ & $0$ & $0$ & $\frac{60}{105}$ & $\frac{21}{105}$
& $0$ & $\frac{24}{105}$ & $0$
\end{tabular}
\end{center}

Consider the following payoffs at time $t = 1$
\begin{center}
\begin{tabular}{c|ccc}
$\Omega$ & $1$ & $2$ & $3$\\
\hline
\\[-1.5ex]
$X$ & $20$ & $10$ & $10$\\[1ex]
$Y$ & $10$ & $10$ & $20$\\
\end{tabular}
\end{center}

We have that
\begin{eqnarray*}
\ups(X) &=& \min\{\EE_{Q_{1,3}}(X), \EE_{Q_{2,3}}(X)\} = 10,\\
\ops(X) &=& \max\{\EE_{Q_{1,3}}(X), \EE_{Q_{2,3}}(X)\} = 12,\\
\ups(Y) &=& \min\{\EE_{Q_{1,3}}(Y), \EE_{Q_{2,3}}(Y)\} =
\frac{110}{7} \approx 15.7,\\
\ops(Y) &=& \max\{\EE_{Q_{1,3}}(Y), \EE_{Q_{2,3}}(Y)\} = 18,
\end{eqnarray*}
so, we can consider the price assessment $\pi(S_1) = 20$, $\pi(X) = 11$, $\pi(Y) = 17$. It holds that the partial price assessments $\{\pi(S_1),\pi(X)\}$ and $\{\pi(S_1),\pi(Y)\}$ are arbitrage-free, while the global price assessment $\{\pi(S_1),\pi(X),\pi(Y)\}$ is not, as there is no $Q \in \Q$ such that
$\pi(S_1) = \EE_Q(S_1)$,
$\pi(X) = \EE_Q(X)$,
$\pi(Y) = \EE_Q(Y)$.
\hfill$\blacklozenge$
\end{example}

One of the main hypotheses underlying the one-period $n$-nomial market model is the absence of frictions that, together with the no-arbitrage principle, imply the linearity of the price functional. Nevertheless, as largely acknowledged in the literature (see, e.g., \cite{amihud1,amihud2}) real markets show frictions, mainly in the form of bid-ask spreads, that translate in the non-linearity of the price functional.

Since we have a set of equivalent martingale measures $\Q$, we could look for a suitable closed subset $\Q' \subseteq \Q$ to define a {\it lower pricing rule} as a discounted lower expectation, in a way to allow frictions in the market. In the literature, several papers investigated the problem of pricing using lower/upper expectation functionals (see, e.g., \cite{bensaid,elkaroui,jouni}).
The choice of $\Q'$ is not free of issues since a reasonable criterion should be provided. The most natural way to get $\Q'$ is to consider a finite $\G \subset \RR^\Omega$, and a {\it lower price assessment} $\upi:\G \to \RR$.
Here, the problem is to look for a closed $\Q' \subseteq \Q$ such that
$$
\upi(X) = \min_{Q \in \Q'} (1+r)^{-1}\EE_Q(X), \mbox{ for every $X \in \G$}.
$$
A first (trivial) constraint for $\upi$ is, for every $X \in \G$,
$$
\ups(X) < \upi(X)  < \ops(X), \mbox{ if $\ups(X) < \ops(X)$},
$$
and $\upi(X) = \ups(X) = \ops(X)$ otherwise,
which, however, does not assure the existence of such a $\Q'$, as shown in the following example.

\begin{example}
Let $\Omega$, $m_1$, $m_2$, $m_3$, $1+r$, $S_0$, $X$ and $Y$ as in Example~\ref{ex:assessment}. Consider the lower price assessment $\upi(S_1) = 20$, $\upi(X) = 11$ and $\upi(Y) = 17$. We have that there is no closed subset $\Q' \subseteq \Q$ such that the corresponding discounted lower expectation functional agrees with $\upi$, in fact
the following system
$$
\left\{
\begin{array}{ll}
q_1 + q_2 + q_3 = 1,\\[1ex]
4q_1 + 2q_2 + \frac{q_3}{4} = 1,\\[1ex]
20 q_1 + 10 q_2 + 10 q_3 = 11,\\[1ex]
10 q_1 + 10 q_2 + 20 q_3 \ge 17,\\[1ex]
q_k \ge 0, & k=1,2,3,
\end{array}
\right.
$$
is not compatible. Notice that the constraint related to $\upi(S_1) = 20$ is not reported since it is implied by the second equation.

We stress that, more generally, for the above assessment there is no closed subset $\mathcal{Q}'' \subseteq \MM(\Omega,\F)$ whose corresponding discounted lower expectation functional agrees with $\upi$. To see this, it is sufficient to consider the above system and relax  the second constraint in a greater than or equal to  constraint, as this result in an incompatible system.
\hfill$\blacklozenge$
\end{example}

Instead of looking for a closed $\Q' \subseteq \Q$, we could try to derive a lower pricing rule from the lower envelope $\underline{Q}$, which has been proved to be a belief function. The most natural way to get a lower pricing rule is to consider a discounted Choquet expectation derived from the ``risk-neutral'' belief function $\underline{Q}$. We stress that, working directly in the framework of belief functions allows to incorporate ``naturally'' frictions in the market, nevertheless, for such a lower pricing rule to be acceptable the classical notion of arbitrage must be generalized. This will be the objective of the next section.

\section{A generalized no-arbitrage principle}
\label{sec:gnoarb}
In this section we consider a finite measurable space $(\Omega,\F)$, with $\Omega = \{1,\ldots,n\}$ and $\F = \mathcal{P}(\Omega)$, endowed with a belief function $Bel$ encoding the market beliefs.
Throughout this section we assume $Bel(A)> 0$, for every $A \in \F \setminus \{\emptyset\}$. Such a belief function $Bel$ plays the same role of the ``real-world'' probability measure $P$ in the classical formulation of a one-period market model (see, e.g., \cite{delbschac}). For this, $Bel$ can be dubbed as ``real-world'' belief function.

\begin{definition}
Given two belief functions $Bel, \widehat{Bel}$ on $\F$, we say that $\widehat{Bel}$ is {\it equivalent} to $Bel$, in symbol $\widehat{Bel} \sim Bel$, if $Bel(A) = 0 \Longleftrightarrow \widehat{Bel}(A) = 0$, for every $A \in \F$.
\end{definition}
 Let us stress that, since $Bel$ is positive on $\F \setminus \{\emptyset\}$, $\widehat{Bel} \sim Bel$ if and only if its M\"obius inverse $\widehat{\mu}$ is positive over the singletons.

 Also in this case, we refer to the filtration $\{\F_0,\F_1\}$ with $\F_0 = \{\emptyset,\Omega\}$ and $\F_1 = \F$.

We still consider a one-period market model related to times $t = 0$ and $t = 1$ where there is a risk-free bond assuring the return $1+r > 0$. Such a bond has price $B_0 = 1$ at time $t = 0$ and payoff $B_1 = 1+r$ at time $t =1$. Here, the goal is to allow frictions in the market by considering, for a random variable $X \in \RR^\Omega$, a lower price $\underline{\pi}(X)$ at time $t = 0$ and, if available, a corresponding upper price $\overline{\pi}(X)$, with $\underline{\pi}(X) \le \overline{\pi}(X)$, to be interpreted as bid-ask prices. In the case of the risk-free bond we assume absence of frictions, meaning that the lower price coincides with the upper price $\underline{\pi}(B_1) = \overline{\pi}(B_1) = B_0$, thus we simply call it price.

We consider a finite non-empty collection of random variables
\begin{equation}
\G = \{S^1_1,\ldots,S^m_1\} \subset \RR^\Omega,
\end{equation}
expressing random payoffs at time $t = 1$ and a lower price assessment $\upi:\G \to \RR$ related to time $t = 0$. In analogy with the classical formulation of no-arbitrage pricing \cite{delbschac}, we do not require the risk-free bond to be part of $\G$ as it possesses a special role being used as num\'{e}raire.

Our aim is to determine a necessary and sufficient condition for the existence of a belief function $\widehat{Bel} \sim Bel$ such that, for $k = 1,\ldots,m$, it holds that
$$
(1+r)^{-1}\CC_{\widehat{Bel}}(S^k_1) = \upi(S^k_1).
$$
By the positive homogeneity property of the Choquet integral (see, e.g., \cite{grabisch}),
$$
(1+r)^{-1}\CC_{\widehat{Bel}}(S^k_1) =\CC_{\widehat{Bel}}( (1+r)^{-1}S^k_1),
$$
thus we can consider the discounted payoff $\tilde{S}^k_1 = (1+r)^{-1}S^k_1$, for $k = 1,\ldots,m$, and write
\begin{equation}
\CC_{\widehat{Bel}}(\tilde{S}^k_1) = \upi(S^k_1).
\end{equation}

Here we assume to know only the lower price of every $S_1^k$, for $k = 1,\ldots,m$. This is not restrictive since, if also the upper price assessment $\overline{\pi}:\G \to \RR$ is available, then the problem can be reformulated by considering
\begin{equation}
\G' = \{S_1^1,\ldots,S_1^m,-S_1^1,\ldots,-S_1^m\}
\end{equation}
together with $\underline{\pi}':\G' \to \RR$ such that,
for $k = 1,\ldots,m$,
\begin{equation}
\underline{\pi}'(S_1^k) = \underline{\pi}(S_1^k)
\quad \mbox{and} \quad
\underline{\pi}'(-S_1^k) = -\overline{\pi}(S_1^k).
\end{equation}

As usual, a {\it portfolio} is a vector $\blambda = (\lambda_1,\ldots,\lambda_m)^T  \in \RR^m$, whose components express the number of units bought/sold of every contract related to the random payoffs in $\G$.

Here, we assume the {\it partially resolving uncertainty} principle proposed by  Jaffray \cite{Jaffray-Bel} according to which the agent may only acquire the information that an event $B\neq \emptyset$ occurs, without knowing which is the true state of the world $\omega \in B$. Further, we assume that the agent is systematically pessimistic in his/her quantitative evaluations. As such, both in computing his/her (discounted) payoff related to a portfolio of securities and in the corresponding gain, the agent considers all non-impossible events in $\U = \F \setminus\{\emptyset\}$ further, for every $X \in \RR^\Omega$, he/she considers the corresponding $X^\l \in \RR^\U$ built taking minima of $X$ as in \eqref{eq:lower-gen}. This is in contrast with the principle of {\it completely resolving uncertainty} which is usually tacitly adopted and amounts in assuming that the agent will always acquire the information on the true state of the world $\omega \in \Omega$.

Working under partially resolving uncertainty, the final (discounted) payoff of the portfolio is the function $Z_\blambda: \U \to \RR$ defined, for every $B \in \U$, as
\begin{equation}
Z_\blambda(B) = \sum_{k = 1}^m \lambda_k(\tilde{S}^k_1)^\l(B),
\end{equation}
while we interpret the quantity $\pi_\blambda= \sum_{k = 1}^m \lambda_k\upi(S^k_1)$ as the hypothetical price at time $t = 0$ of the portfolio that we would have if we were in a situation of completely resolving uncertainty. Hence, we can define the function $G_\blambda:\U \to \RR$ setting, for every $B \in \U$,
\begin{equation}
G_\blambda(B) = Z_\blambda(B) - \pi_\blambda = \sum_{k = 1}^m \lambda_k\left((\tilde{S}^k_1)^\l(B) - \upi(S^k_1)\right),
\end{equation}
that can be interpreted as a random gain under partially resolving uncertainty.

\begin{theorem}
\label{th:nodutch}
The following conditions are equivalent:
\begin{itemize}
\item[\it (i)] there exists a belief function $\widehat{Bel}$ such that $\CC_{\widehat{Bel}}(\tilde{S}^k_1) = \upi(S^k_1)$, for $k = 1,\ldots,m$;
\item[\it (ii)] for every $\mbox{\boldmath $\lambda$} = (\lambda_1,\ldots,\lambda_m)^T\in \RR^m$ it holds that
$$
\min_{B \in \U} G_\blambda(B) \le 0 \le \max_{B \in \U} G_\blambda(B).
$$
\end{itemize}
\end{theorem}
\begin{proof}
The proof can be obtained applying Theorem~4.1 in \cite{cpv-bumi}, working with the dual capacity of $\widehat{Bel}$, which is a plausibility function. Here we provide a direct proof for the sake of completeness.

Fix an enumeration of $\U = \{B_1,\ldots,B_{2^n - 1}\}$. Condition {\it (i)} is equivalent to the solvability of the following system
$$
\left\{
\begin{array}{ll}
{\bf A}{\bf x} = {\bf b},\\
{\bf x} \ge {\bf 0},
\end{array}
\right.
$$
where ${\bf x} = (\widehat{\mu}(B_1),\ldots,\widehat{\mu}(B_{2^n-1}))^T \in \RR^{(2^n - 1)}$ is an unknown column vector, ${\bf A} \in \RR^{(m+1) \times (2^n - 1)}$ is the coefficient matrix with
$$
{\bf A} =
\left(
\begin{array}{ccc}
(\tilde{S}^1_1)^\l(B_1) & \cdots & (\tilde{S}^1_1)^\l(B_{2^n-1})\\
\vdots & & \vdots\\
(\tilde{S}^m_1)^\l(B_1) & \cdots & (\tilde{S}^m_1)^\l(B_{2^n-1})\\
\1_\Omega^\l(B_1) & \cdots & \1_\Omega^\l(B_{2^n-1})
\end{array}
\right),
$$
and ${\bf b} = (\upi(S^1_1),\ldots,\upi(S^m_1),1)^T \in \RR^{(m+1)}$.

By Farkas' lemma \cite{mangasarian}, the system above is compatible if and only if the following system is not compatible
$$
\left\{
\begin{array}{ll}
{\bf A}^T{\bf y} \le {\bf 0},\\
{\bf b}^T{\bf y} > 0,
\end{array}
\right.
$$
where ${\bf y} = (\lambda_1,\ldots,\lambda_m,\lambda_{m+1})^T \in \RR^{(m+1)}$ is an unknown column vector. It holds that ${\bf A}^T{\bf y} \in \RR^{(2^n - 1)}$ and, for $i=1,\ldots,2^n - 1$, the $i$th component of constraint ${\bf A}^T{\bf y} \le {\bf 0}$ is
$$
\sum_{k=1}^m \lambda_k (\tilde{S}^k_1)^\l (B_i) + \lambda_{m+1} \le 0,
$$
moreover, subtracting the positive quantity ${\bf b}^T{\bf y}$ we get
$$
\sum_{k=1}^m \lambda_k\left((\tilde{S}^k_1)^\l(B_i) - \upi(S_1^k)\right) < 0.
$$
Thus, condition {\it (i)} is equivalent to the existence of $i \in \{1,\ldots,2^n - 1\}$ such that the above inequality does not hold, which, in turn, is equivalent to {\it (ii)}.
\hfill$\square$
\end{proof}

The above theorem says that, working under partially resolving uncertainty, in order to have a discounted totally monotone Choquet expectation representation of the lower price assessment $\upi$, it is necessary and sufficient that every portfolio $\mbox{\boldmath $\lambda$}$ does not give rise to a sure loss or a sure gain over $\U$. In other terms, the above condition can be considered a {\it generalized avoiding Dutch book condition}, working under partially resolving uncertainty.
Nevertheless, the condition {\it (ii)} of Theorem~\ref{th:nodutch} does not assure that $\widehat{Bel} \sim Bel$, that is we do not have any guarantee that $\widehat{Bel}(A) > 0$, for every $A \in \F \setminus \{\emptyset\}$.

The following theorem provides a necessary and sufficient condition for the existence of an equivalent belief function positive on the entire $\F \setminus \{\emptyset\}$ that can be dubbed ``risk-neutral'' belief function.

Such theorem is the analog of the {\it first fundamental theorem of asset pricing}, formulated in the Dempster-Shafer theory of evidence.

\begin{theorem}
\label{th:noarb}
The following conditions are equivalent:
\begin{itemize}
\item[\it (i)] there exists a belief function $\widehat{Bel} \sim Bel$, i.e., $\widehat{Bel}(A) > 0$, for every $A \in \F \setminus \{\emptyset\}$, such that $\CC_{\widehat{Bel}}(\tilde{S}^k_1) = \upi(S^k_1)$, for $k = 1,\ldots,m$;
\item[\it (ii)] for every $\mbox{\boldmath $\lambda$} = (\lambda_1,\ldots,\lambda_m)^T  \in \RR^{m}$ none of the following conditions holds:
\begin{itemize}
\item[\it (a)] $Z_\blambda(\{i\}) = 0$, for $i = 1,\ldots,n$,  $Z_\blambda(B) \ge 0$, for all $B \in \U \setminus \{\{i\} \,:\, i \in \Omega\}$ and $\pi_\blambda < 0$;
\item[\it (b)] $Z_\blambda(\{i\}) \ge 0$, for $i = 1,\ldots,n$, with at least a strict inequality, $Z_\blambda(B) \ge 0$, for all $B \in \U \setminus \{\{i\} \,:\, i \in \Omega\}$, and $\pi_\blambda \le 0$.
\end{itemize}
\end{itemize}
\end{theorem}
\begin{proof}
Since every belief function is completely characterized by its M\"obius inverse, statement {\it (i)} is equivalent to the existence of a non-negative function $\widehat{\mu} : \F \to \RR$ such that
$$
\widehat{\mu}(\emptyset) = 0, \quad \sum_{A \in \F} \widehat{\mu}(A) = 1, \quad \mbox{and} \quad
\widehat{Bel}(A) = \sum_{B \subseteq A} \widehat{\mu}(B),  \quad \mbox{for every $A \in \F$},
$$
further satisfying $\widehat{\mu}(\{i\}) > 0$, for all $i \in \Omega$, and
$$
\CC_{\widehat{Bel}}(\tilde{S}^k_1) = \sum_{B \in \U} (\tilde{S}^k_1)^\l(B) \widehat{\mu}(B) = \upi(S^k_1), \quad \mbox{for $k = 1,\ldots,m$}.
$$

Fix an enumeration of $\U = \{B_1,\ldots,B_{2^n -1}\}$ such that $B_i = \{i\}$, for $i=1,\ldots,n$, and consider the matrices ${\bf A} \in \RR^{(2(m+1) + 2^n-(n+1)) \times (2^n-1)}$ and ${\bf B}  \in \RR^{n \times (2^n - 1)}$ defined as
$$
{\bf A} = \left(
\begin{array}{c}
{\bf C}\\
\hline
\left.{\bf O}_1\right| -{\bf I}_{(2^n - (n+1))}
\end{array}
\right)
\quad \mbox{and} \quad
{\bf B} = \left(-{\bf I}_n \left| {\bf O}_2\right)\right.,
$$
where ${\bf C} \in \RR^{2(m+1) \times (2^n-1)}$ is defined as
$$
{\bf C} =
\left(
\begin{array}{ccc}
(\tilde{S}_1^1)^\l(B_1) & \cdots & (\tilde{S}_1^1)^\l(B_{2^n-1})\\
-(\tilde{S}_1^1)^\l(B_1) & \cdots & -(\tilde{S}_1^1)^\l(B_{2^n-1})\\
\vdots & & \vdots\\
(\tilde{S}_1^m)^\l(B_1) & \cdots & (\tilde{S}_1^m)^\l(B_{2^n-1})\\
-(\tilde{S}_1^m)^\l(B_1) & \cdots & -(\tilde{S}_1^m)^\l(B_{2^n-1})\\
\1_\Omega^\l(B_1) & \cdots & \1_\Omega^\l(B_{2^n-1})\\
-\1_\Omega^\l(B_1) & \cdots & -\1_\Omega^\l(B_{2^n-1})\\
\end{array}
\right),
$$
in which ${\bf I}_{(2^n - (n+1))} \in \RR^{(2^n - (n+1)) \times (2^n-(n+1))}$ and  ${\bf I}_n \in \RR^{n \times n}$ are identity matrices, and ${\bf O}_1 \in \RR^{(2^n -(n+1)) \times n}$ and ${\bf O}_2 \in \RR^{n \times (2^n - (n+1))}$ are null matrices.
Take the vector
$$
{\bf b} = (\upi(S_1^1),-\upi(S_1^1),\ldots,\upi(S_1^m),-\upi(S_1^m),1,-1,0,\ldots,0)^T
$$
with ${\bf b} \in \RR^{(2(m+1) + 2^n - (n+1))}$
and consider the unknown vector
$$
{\bf x} = (\widehat{\mu}(B_1),\ldots,\widehat{\mu}(B_{2^n - 1}))^T
$$
with ${\bf x} \in \RR^{(2^n -1)}$. Condition  {\it (i)} turns out to be equivalent to the solvability of the following system
$$
\left\{
\begin{array}{l}
{\bf A}{\bf x} \le {\bf b},\\
{\bf B}{\bf x} < {\bf 0}.
\end{array}
\right.
$$
By a well-known version of Motzkin's theorem of the alternative (see, e.g., Theorem 1 in \cite{ben-israel}) the above system is solvable if and only if for every ${\bf y} = (y_1,y_1',\ldots,y_m,y_m',y_{m+1},y_{m+1}',\alpha_{n+1},\ldots,\alpha_{2^n-1})^T \in \RR^{(2(m+1) + 2^n - (n+1))}$ and ${\bf z} = (z_1,\ldots,z_n)^T \in \RR^n$ with ${\bf y} \ge {\bf 0}$ and ${\bf z} \ge {\bf 0}$, none of the following conditions holds:
\begin{itemize}
\item ${\bf A}^T{\bf y} + {\bf B}^T {\bf z} = {\bf 0}$, ${\bf z} = {\bf 0}$ and ${\bf b}^T{\bf y} < 0$;
\item ${\bf A}^T{\bf y} + {\bf B}^T {\bf z} = {\bf 0}$, ${\bf z} \neq {\bf 0}$ and ${\bf b}^T{\bf y} \le 0$.
\end{itemize}
In turn, setting $\lambda_k = y_k - y_k'$, for $k = 1,\ldots,m+1$, and considering
$\tilde{\bf y} \in \RR^{((m + 1)+2^n-(n+1))}$,
$\tilde{\bf A} \in \RR^{((m + 1) + 2^n - (n+1)) \times (2^n-1)}$ and $\tilde{\bf b} \in\RR^{((m + 1)+2^n-(n+1))}$, with
$$
\tilde{\bf y} = (\lambda_1,\ldots,\lambda_m,\lambda_{m+1},\alpha_{n+1},\ldots,\alpha_{2^n-1})^T \quad \mbox{such that $\alpha_{n+1},\ldots,\alpha_{2^n-1} \ge 0$,}
$$
$$
\tilde{\bf A} = \left(
\begin{array}{c}
\tilde{\bf C}\\
\hline
\left.{\bf O}_1\right| -{\bf I}_{(2^n - (n+1))}
\end{array}
\right)
\quad \mbox{and} \quad
\tilde{\bf b} = (\upi(S_1^1),\ldots,\upi(S_1^m),1,0,\ldots,0)^T,
$$
where $\tilde{\bf C} \in \RR^{(m+1) \times (2^n - 1)}$ is defined as
$$
\tilde{\bf C} =
\left(
\begin{array}{ccc}
(\tilde{S}_1^1)^\l(B_1) & \cdots & (\tilde{S}_1^1)^\l(B_{2^n-1})\\
\vdots & & \vdots\\
(\tilde{S}_1^m)^\l(B_1) & \cdots & (\tilde{S}_1^m)^\l(B_{2^n-1})\\
\1_\Omega^\l(B_1) & \cdots & \1_\Omega^\l(B_{2^n-1})
\end{array}
\right),
$$
the above conditions can be rewritten as:
\begin{itemize}
\item $\tilde{\bf A}^T\tilde{\bf y} + {\bf B}^T {\bf z} = {\bf 0}$, ${\bf z} = {\bf 0}$ and $\tilde{\bf b}^T\tilde{\bf y} < 0$;
\item $\tilde{\bf A}^T\tilde{\bf y}+ {\bf B}^T {\bf z} = {\bf 0}$, ${\bf z} \neq {\bf 0}$ and $\tilde{\bf b}^T\tilde{\bf y} \le 0$.
\end{itemize}

Denoting $\blambda = (\lambda_1,\ldots,\lambda_m)^T \in \RR^m$, we have that
$$
(\tilde{\bf A}^T\tilde{\bf y} + {\bf B}^T {\bf z})_i
=
\left\{
\begin{array}{ll}
Z_\blambda(B_i) + \lambda_{m+1} - z_i, & \mbox{for $i=1,\ldots,n$},\\[1ex]
Z_\blambda(B_i) + \lambda_{m+1} - \alpha_i, & \mbox{for $i=n+1,\ldots,2^n-1$},
\end{array}
\right.
$$
and further $\tilde{\bf b}^T\tilde{\bf y} = \pi_\blambda + \lambda_{m+1}$.

Hence, for every $\blambda = (\lambda_1,\ldots,\lambda_m)^T \in \RR^m$ and $\lambda_{m+1} \in \RR$, the above conditions can be rewritten as
\begin{description}
\item[\it (a')] $Z_\blambda(\{i\}) + \lambda_{m+1} = 0$, for $i = 1,\ldots,n$,  $Z_\blambda(B) + \lambda_{m+1} \ge 0$, for all $B \in \U \setminus \{\{i\} \,:\, i \in \Omega\}$ and $\pi_\blambda + \lambda_{m+1} < 0$;
\item[\it (b')] $Z_\blambda(\{i\}) + \lambda_{m+1} \ge 0$, for $i = 1,\ldots,n$, with at least a strict inequality, $Z_\blambda(B) + \lambda_{m+1} \ge 0$, for all $B \in \U \setminus \{\{i\} \,:\, i \in \Omega\}$, and $\pi_\blambda + \lambda_{m+1} \le 0$.
\end{description}
Finally, for every $\blambda = (\lambda_1,\ldots,\lambda_m)^T \in \RR^m$ and $\lambda_{m+1} \in \RR$, none between {\it (a')} and {\it (b')} holds if and only if for every $\blambda = (\lambda_1,\ldots,\lambda_m)^T \in \RR^m$, none between {\it (a)} and {\it (b)} holds. Indeed, the first implication is obtained taking $\lambda_{m+1} = 0$, while the converse is obtained proceeding assuming that one between {\it (a)} and {\it (b)} holds and proceeding by contradiction.
\hfill$\square$
\end{proof}

Recall that we interpret $\pi_\blambda$ as the hypothetical price of the portfolio $\blambda$ as if we were in a situation of completely resolving uncertainty.
In this light, conditions {\it (ii.a)} and {\it (ii.b)} of previous theorem can be interpreted as two generalized forms of arbitrage, working under partially resolving uncertainty.
Avoiding condition {\it (ii.a)} assures that we cannot find a portfolio $\blambda$ whose hypothetical price $\pi_\blambda$ is negative (that is we are paid for it), resulting in a uniformly non-negative payoff $Z_\blambda$ in all the possible events in $\U$, with null value on the singletons (i.e., on those events where we have completely resolving uncertainty). Avoiding condition  {\it (ii.b)} assures that we cannot find a portfolio $\blambda$ whose hypothetical price $\pi_\blambda$ is negative or null (that is we are paid or we do not pay anything for it), resulting in a uniformly non-negative payoff $Z_\blambda$ in all the possible events in $\U$, with at least a strictly positive value on the singletons (i.e., on those events where we have completely resolving uncertainty).

It is immediate to see that the generalized no-arbitrage principle expressed by statement {\it (ii)} of Theorem~\ref{th:noarb} implies the generalized avoiding Dutch book condition in statement {\it (ii)} of Theorem~\ref{th:nodutch}.

Let us stress that the generalized no-arbitrage principle of Theorem~\ref{th:noarb} is actually weaker than the classical no-arbitrage principle. This is due to the fact that if a portfolio $\blambda$ gives rise to a generalized arbitrage of the form {\it (ii.a)} or {\it (ii.b)} then it also gives rise to a classical arbitrage, while a portfolio $\blambda$ giving rise to a classical arbitrage does not generally give rise to a generalized arbitrage.

Let us stress that the price functional determined by the discounted Choquet expectation with respect to a $\widehat{Bel}$ with M\"obius inverse $\widehat{\mu}$ as in Theorem~\ref{th:noarb} is generally not linear. In particular, we have that
\begin{eqnarray}
\sum_{B \in \U} Z_\blambda(B) \widehat{\mu}(B) &=&
\sum_{B \in \U} \left(\sum_{k=1}^m \lambda_k (\tilde{S}_1^k)^\l(B)\right) \widehat{\mu}(B)\\\nonumber
&=&
\sum_{k=1}^m \lambda_k \left(\sum_{B \in \U} (\tilde{S}_1^k)^\l(B) \widehat{\mu}(B) \right)\\\nonumber
&=& \sum_{k=1}^m \lambda_k \CC_{\widehat{Bel}}(\tilde{S}_1^k) = \sum_{k=1}^m \lambda_k \upi(S_1^k) = \pi_\blambda,
\end{eqnarray}
nevertheless, considering the random variable
$\sum_{k=1}^m\lambda_k \tilde{S}_1^k \in \RR^\Omega$, in general we have that
$$
\CC_{\widehat{Bel}}\left(\sum_{k=1}^m\lambda_k \tilde{S}_1^k\right)
\neq
\sum_{k=1}^m\CC_{\widehat{Bel}}(\lambda_k \tilde{S}_1^k)
\quad
\mbox{and}
\quad
\sum_{k=1}^m\CC_{\widehat{Bel}}(\lambda_k \tilde{S}_1^k)
\neq
\sum_{k=1}^m \lambda_k \CC_{\widehat{Bel}}(\tilde{S}_1^k).
$$
Clearly, in the above formulas we have equalities in case $\widehat{Bel}$ reduces to a probability measure. On the other hand, in the particular case $\tilde{S}_1^h,\tilde{S}_1^k$ are comonotonic and $\lambda_1,\lambda_2 \ge 0$, it holds that
$$
\CC_{\widehat{Bel}}\left(\lambda_1 \tilde{S}_1^h + \lambda_2 \tilde{S}_1^k\right) = \lambda_1 \CC_{\widehat{Bel}}(\tilde{S}_1^h) + \lambda_2 \CC_{\widehat{Bel}}(\tilde{S}_1^k) = \lambda_1 \upi(S_1^h) + \lambda_2 \upi(S_1^k).
$$
Furthermore, denoting by $\widehat{Pl}$ the dual plausibility function of $\widehat{Bel}$, we have that
for a generic random variable $X \in \RR^\Omega$, it holds that
\begin{equation}
(1+r)^{-1}\CC_{\widehat{Bel}}(X) \le (1+r)^{-1}\CC_{\widehat{Pl}}(X),
\end{equation}
i.e., the two values above should be interpreted as lower and upper prices.

The following example shows a lower price assessment violating the generalized no-arbitrage principle expressed in Theorem~\ref{th:noarb}.

\begin{example}
\label{ex:nodutch}
Let $\Omega = \{1,2,3,4\}$, $\F = {\cal P}(\Omega)$ and consider three contracts whose payoffs in euros at time $t = 1$ are
\begin{center}
\begin{tabular}{c|cccc}
$\Omega$ & $1$ & $2$ & $3$ & $4$\\
\hline
$S_1^1$ & $10$ & $10$ & $20$ & $20$\\
$S_1^2$ & $0$ & $10$ & $0$ & $10$\\
$S_1^3$ & $10$ & $30$ & $20$ & $40$
\end{tabular}
\end{center}
Assume that the lower prices at time $t = 0$ are fixed to $\upi(S_1^1) = 15$, $\upi(S_1^2) = 5$ and $\upi(S_1^3) = 20$ and that the risk-free interest rate is $r = 0$, so we have $\tilde{S}_1^k = S_1^k$, for $k = 1,2,3$.

This lower price assessment violates the generalized no-arbitrage principle of Theorem~\ref{th:noarb} as, in particular, it violates the generalized avoiding Dutch book condition expressed in Theorem~\ref{th:nodutch}. Indeed, every belief function $\widehat{Bel}$ on $\F$ induces a Choquet expectation functional on $\RR^\Omega$ which is positively homogeneous and superadditive, therefore, assuming $\CC_{\widehat{Bel}}(S_1^k) = \upi(S_1^k)$, for $k=1,2,3$, it should be, as $S_1^3 = S_1^1 + 2S_1^2$,
$$
\CC_{\widehat{Bel}}(S_1^3) =
\CC_{\widehat{Bel}}(S_1^1 + 2S_1^2) \ge
\CC_{\widehat{Bel}}(S_1^1) + 2\CC_{\widehat{Bel}}(S_1^2) = 25.
$$

Denoting $\U = \F \setminus \{\emptyset\}$ and omitting braces and commas to have a lighter set notation, if we consider the portfolio $\blambda = (-1,-2,1)^T$ we have that $\pi_\blambda = -5$ and
\begin{center}
\begin{tabular}{c|ccccccccccccccc}
$\U$ & $1$ & $2$ & $3$ & $4$ &
$12$ & $13$ & $14$ & $23$ & $24$ & $34$ &
$123$ & $124$ & $134$ & $234$ & $1234$
\\
\hline
$(S_1^1)^\l$ & $10$ & $10$ & $20$ & $20$ &
$10$ & $10$ & $10$ & $10$ & $10$ & $20$ &
$10$ & $10$ & $10$ & $10$ & $10$\\
$(S_1^2)^\l$ & $0$ & $10$ & $0$ & $10$ &
$0$ & $0$ & $0$ & $0$ & $10$ & $0$ &
$0$ & $0$ & $0$ & $0$ & $0$\\
$(S_1^3)^\l$ & $10$ & $30$ & $20$ & $40$ &
$10$ & $10$ & $10$ & $20$ & $30$ & $20$ &
$10$ & $10$ & $10$ & $20$ & $10$
\\
\hline
$Z_\blambda$ & $0$ & $0$ & $0$ & $0$ &
$0$ & $0$ & $0$ & $10$ & $0$ & $0$ &
$0$ & $0$ & $0$ & $10$ & $0$\\
$G_\blambda$ & $5$ & $5$ & $5$ & $5$ &
$5$ & $5$ & $5$ & $15$ & $5$ & $5$ &
$5$ & $5$ & $5$ & $15$ & $5$
\end{tabular}
\end{center}
Hence, since $\min\limits_{B \in \U} G_\blambda(B) > 0$, the generalized avoiding Dutch book condition is not satisfied, therefore there is no belief function $\widehat{Bel}$ such that $\CC_{\widehat{Bel}}$ agrees with the assessed lower prices. Moreover, the same $\blambda$ shows that we have a generalized arbitrage in the form of {\it (ii.a)} of Theorem~\ref{th:noarb}, since $Z_\blambda(\{i\}) = 0$, for $i = 1,\ldots,4$,  $Z_\blambda(B) \ge 0$, for all $B \in \U \setminus \{\{i\} \,:\, i \in \Omega\}$ and $\pi_\blambda < 0$.
\hfill$\blacklozenge$
\end{example}

The following example shows a lower price assessment violating the classical no-arbitrage principle but not the generalized no-arbitrage principle.

\begin{example}
Let $\Omega$, $\F$, $r$, and $S_1^1,S_1^2,S_1^3$ as in Example~\ref{ex:nodutch}. Consider the lower price assessment $\upi(S_1^1) = 15$, $\upi(S_1^2) = 5$ and $\upi(S_1^3) = 26$. Such an assessment violates the classical no-arbitrage principle, indeed, every probability measure $Q$ on $\F$ gives rise to a positive, linear and normalized functional $\EE_Q$ on $\RR^\Omega$. Hence, assuming $\EE_Q(S_1^k) = \upi(S_1^k)$, for $k=1,2,3$, it should be, as $S_1^3 = S_1^1 + 2S_1^2$,
$$
\EE_Q(S_1^3) = \EE_Q(S_1^1 + 2S_1^2) = \EE_Q(S_1^1) + 2\EE_Q(S_1^2) = 25.
$$

On the other hand, there exists a belief function $\widehat{Bel}$ on $\F$ which is strictly positive on $\F \setminus \{\emptyset\}$, whose corresponding Choquet expectation functional $\CC_{\widehat{Bel}}$ agrees with the given lower price assessment. For instance, we can take the $\widehat{Bel}$ whose M\"obius inverse $\widehat{\mu}$ is such that
\begin{center}
\begin{tabular}{c|ccccccccccccccc}
$\U$ & $1$ & $2$ & $3$ & $4$ &
$12$ & $13$ & $14$ & $23$ & $24$ & $34$ &
$123$ & $124$ & $134$ & $234$ & $1234$
\\
\hline
$(S_1^1)^\l$ & $10$ & $10$ & $20$ & $20$ &
$10$ & $10$ & $10$ & $10$ & $10$ & $20$ &
$10$ & $10$ & $10$ & $10$ & $10$\\
$(S_1^2)^\l$ & $0$ & $10$ & $0$ & $10$ &
$0$ & $0$ & $0$ & $0$ & $10$ & $0$ &
$0$ & $0$ & $0$ & $0$ & $0$\\
$(S_1^3)^\l$ & $10$ & $30$ & $20$ & $40$ &
$10$ & $10$ & $10$ & $20$ & $30$ & $20$ &
$10$ & $10$ & $10$ & $20$ & $10$
\\
\hline
$\widehat{\mu}$ & $\frac{2}{10}$ & $\frac{1}{10}$ & $\frac{1}{10}$ & $\frac{4}{10}$ &
$\frac{1}{10}$ & $0$ & $0$ & $\frac{1}{10}$ & $0$ & $0$ &
$0$ & $0$ & $0$ & $0$ & $0$
\end{tabular}
\end{center}

For such a $\widehat{Bel}$ we have that
\begin{eqnarray*}
\CC_{\widehat{Bel}}(S_1^1) &=& 10 \cdot \frac{2}{10} + (10+20+10+10) \cdot \frac{1}{10} + 20 \cdot \frac{4}{10} = 15,\\
\CC_{\widehat{Bel}}(S_1^2) &=& 10 \cdot \frac{1}{10} + 10 \cdot \frac{4}{10} = 5,\\
\CC_{\widehat{Bel}}(S_1^3) &=& 10 \cdot \frac{2}{10} + (30+20+10+20) \cdot \frac{1}{10} + 40 \cdot \frac{4}{10} = 26.
\end{eqnarray*}

Hence, by Theorem~\ref{th:noarb} we cannot find a portfolio $\blambda$ giving rise to a generalized arbitrage in the form of {\it (ii.a)} or {\it (ii.b)}. On the other hand, the portfolio $\blambda = (1,2,-1)^T$ gives rise to a classical arbitrage since $\pi_\blambda = -1 < 0$ and
\begin{center}
\begin{tabular}{c|cccc}
$\Omega$ & $1$ & $2$ & $3$ & $4$
\\
\hline
$S_1^1$ & $10$ & $10$ & $20$ & $20$\\
$S_1^2$ & $0$ & $10$ & $0$ & $10$\\
$S_1^3$ & $10$ & $30$ & $20$ & $40$
\\
\hline
$\sum_{k=1}^3\lambda_k S_1^k$ & $0$ & $0$ & $0$ & $0$
\end{tabular}
\end{center}
\hfill$\blacklozenge$
\end{example}

A non-linear pricing rule defined as a discounted Choquet expectation with respect to a concave capacity and satisfying a form of put-call parity has been axiomatically characterized, respectively, in \cite{ckl} and \cite{cmm}. Such a functional, that can be interpreted as an upper pricing rule, allows to model frictions in the market. Our generalized no-arbitrage condition is equivalent to the existence of a (possibly not unique) completely monotone discounted Choquet expectation that can be interpreted as a lower pricing rule still allowing for frictions in the market.

Let us stress that, if $\{\1_B \,:\, B \in \U\} \subseteq \G$ and $\upi:\G \to \RR$ satisfies the generalized no-arbitrage principle, then there exists a unique $\widehat{Bel}$, positive on $\F \setminus \{\emptyset\}$, such that the corresponding discounted Choquet expectation functional on $\RR^\Omega$ agrees with $\upi$. The payoffs in $\{\1_B \,:\, B \in \U\}$ can be considered as {\it generalized Arrow-Debreu securities} (see, e.g., \cite{dybvig,pliska,cerny}), working under partially resolving uncertainty.

\section{Equivalent inner approximating martingale belief functions}
\label{sec:inner}
We turn back to the lower envelope $\underline{Q}$ of the class $\Q$ of equivalent martingale measures induced by the $n$-nomial market model characterized in Section~\ref{sec:emm}. Recall that in this context we have only one risky asset whose price process is $\{S_0,S_1\}$. As already pointed out in Corollary~\ref{cor:lower}, $\underline{Q}$ is a belief function that we know is not positive over $\F\setminus\{\emptyset\}$, for $n > 2$. It is easily seen that every equivalent martingale measure $Q \in \Q$ is a belief function, positive on $\F \setminus \{\emptyset\}$, nevertheless, the choice of a particular $Q_0$ in the class $\Q$ is a problematic task, as one needs to provide a reasonable choice criterion. For instance, a possible choice is $Q_0 = \frac{1}{|\ext(\cl(\Q))|}\sum_{Q \in \ext(\cl(\Q))} Q$, which belongs to $\Q$ since it is a strict convex combination of elements of $\ext(\cl(\Q))$.

Once an equivalent martingale measure $Q_0$ has been chosen, some of the information contained in the class $\Q$ can be preserved if we consider (see, e.g., \cite{huber,walley-libro}) the {\it $\epsilon$-contamination} of $Q_0$ with respect to $\cl(\Q)$, where $\epsilon \in (0,1)$. This amounts to consider the closed subset of $\Q$ given by
\begin{equation}
\Q_\epsilon = \{Q \in \Q \,:\, Q = (1-\epsilon)Q_0 + \epsilon Q', Q' \in \cl(\Q)\},
\end{equation}
whose lower envelope $\underline{Q}_\epsilon = \min \Q_\epsilon$ is defined, for every $A \in \F$, as
\begin{equation}
\underline{Q}_\epsilon(A) = (1-\epsilon)Q_0(A) + \epsilon \underline{Q}(A).
\end{equation}
In particular, since $\underline{Q}_\epsilon$ is the strict convex combination of the two belief functions $Q_0$ and $\underline{Q}$, we have that $\underline{Q}_\epsilon$ is a belief function which, in turn, is strictly positive over $\F \setminus \{\emptyset\}$, as $Q_0$ is.

The idea is to directly use $\underline{Q}_\epsilon$ in order to derive a lower pricing rule through a discounted Choquet expectation. In the light of previous section, the obtained lower pricing rule would be acceptable if it satisfied the generalized no-arbitrage condition expressed by Theorem~\ref{th:noarb}.
The following example shows that, this is not the case for $n > 2$, and this is due to the fact that $\underline{Q}$ fails to satisfy the generalized avoiding Dutch book condition given in Theorem~\ref{th:nodutch}. Notice that in this context, the ``real-world'' probability $P$ can play the role of the ``real-world'' belief function introduced in Section~\ref{sec:gnoarb}.

Let us stress that, since $\min\limits_{Q \in \cl(\Q)} (1+r)^{-1}\EE_Q(S_1) = S_0$, we can refer to the lower price assessment $\upi(S_1) = S_0$.

\begin{example}
Consider $\Omega = \{1,2,3,4\}$, $m_1 = 4$, $m_2 = 2$, $m_3 = \frac{1}{2}$, $m_4 = \frac{1}{4}$, $1+r = 1$, $\Q$ and $\underline{Q}$ of Example~\ref{ex:lp}, and let $Q_0$ be an arbitrary element of $\Q$. Take $\epsilon \in (0,1)$ and consider the $\epsilon$-contamination class of $Q_0$ with respect to $\cl(\Q)$, whose lower envelope is $\underline{Q}_\epsilon = (1-\epsilon)Q_0 + \epsilon \underline{Q}$.

The lower pricing rule obtained as the discounted Choquet expectation with respect $\underline{Q}_\epsilon$ does not satisfy the generalized no-arbitrage principle expressed by Theorem~\ref{th:noarb}. Indeed, for it to be satisfied it should be
$$
(1+r)^{-1}\CC_{\underline{Q}_\epsilon}(S_1) = S_0,
$$
which is equivalent, since $S_0 = s > 0$, to
$$
\CC_{\underline{Q}_\epsilon}\left(\frac{S_1}{S_0}\right) = 1+r.
$$
Nevertheless, due to linearity of the Choquet integral with respect to the integrating capacity (see, e.g., \cite{denneberg,grabisch}), we have that
\begin{eqnarray*}
\CC_{\underline{Q}_\epsilon}\left(\frac{S_1}{S_0}\right) &=&
\CC_{(1-\epsilon)Q_0 + \epsilon \underline{Q}}\left(\frac{S_1}{S_0}\right)
= (1-\epsilon) \CC_{Q_0}\left(\frac{S_1}{S_0}\right) + \epsilon\CC_{\underline{\QQ}}\left(\frac{S_1}{S_0}\right)\\
&=& (1-\epsilon) + \epsilon \frac{54}{105} < 1+r,
\end{eqnarray*}
since $\CC_{\underline{\QQ}}\left(\frac{S_1}{S_0}\right) = \frac{54}{105}$.
\hfill$\blacklozenge$
\end{example}

A possible way to fulfill the generalized no-arbitrage principle is to look for an inner approximation $\widehat{Bel}$ of $\underline{Q}$ satisfying the generalized avoiding Dutch book condition and then define, for $\epsilon \in (0,1)$,
\begin{equation}
\widehat{Bel}_\epsilon = (1-\epsilon)Q_0 + \epsilon \widehat{Bel}.
\end{equation}
The belief functions $\widehat{Bel}$ and $\widehat{Bel}_\epsilon$ will be referred to as {\it inner approximating martingale belief function} and {\it equivalent inner approximating martingale belief function}, according to the following definition.

\begin{definition}
\label{def:martbel}
A belief function $\widehat{Bel}$ on $\F$ is said:
\begin{itemize}
\item an {\bf inner approximation} for $\underline{Q}$ if, for every $A \in \F$, it holds that
$$
\underline{Q}(A) \le \widehat{Bel}(A);
$$
\item a {\bf martingale belief function} if
$$
(1+r)^{-1}\CC_{\widehat{Bel}}(S_1) = S_0;
$$
\item an {\bf inner approximating martingale belief function} for $\underline{Q}$ if it is both an inner approximation for $\underline{Q}$ and a martingale belief function;
\item an {\bf equivalent inner approximating martingale belief function} for $\underline{Q}$ if it is an inner approximating martingale belief function for $\underline{Q}$ and $\widehat{Bel} \sim P$.
\end{itemize}
\end{definition}

Trivially, we have that every $Q \in \cl(\Q)$ is an inner approximating martingale belief function for $\underline{Q}$, while every $Q \in \Q$ is an equivalent inner approximating martingale belief function for $\underline{Q}$.

Avoiding triviality, our goal is to select a belief function $\widehat{Bel} \in \BB(\Omega,\F)$ which is an inner approximating martingale belief function for $\underline{Q}$ and is closest to $\underline{Q}$ with respect to a suitable distance $d$ defined on the set $\BB(\Omega,\F)$. Following \cite{mmv,mmv2}, two possible choices for the distance $d$ are
\begin{eqnarray}
d_1(Bel_1,Bel_2) &=&
\sum\limits_{A \in \F} |Bel_1(A) - Bel_2(A)|,\\
d_2(Bel_1,Bel_2) &=&
{\sum\limits_{A \in \F} (Bel_1(A) - Bel_2(A))^2}.
\end{eqnarray}

Thus, for a fixed distance $d$, an optimal inner approximating martingale belief function $\widehat{Bel}$ for $\underline{Q}$ can be found by solving the following optimization problem:
\begin{equation}
\label{eq:op}
\begin{array}{c}
\minimize d(\widehat{Bel},\underline{Q})\\[1ex]
\mbox{subject to:}\\[1ex]
\left\{
\begin{array}{ll}
\widehat{Bel}(A) \ge \underline{Q}(A), & \mbox{for every $A \in \F$},\\[1ex]
(1+r)^{-1}\CC_{\widehat{Bel}}(S_1) = S_0,\\[1ex]
\widehat{Bel} \in \BB(\Omega,\F).
\end{array}
\right.
\end{array}
\end{equation}
Denote as before $\U = \F \setminus \{\emptyset\}$. The searched $\widehat{Bel}$ is completely characterized by its M\"obius inverse $\widehat{\mu}$ that must satisfy $\widehat{\mu}(\emptyset) = \widehat{Bel}(\emptyset) = \underline{Q}(\emptyset) = 0$. Moreover, since $S_0 = s > 0$, it holds that
$$
(1+r)^{-1}\CC_{\widehat{Bel}}(S_1) = S_0
$$
is equivalent to
$$
\CC_{\widehat{Bel}}\left(\frac{S_1}{S_0}\right) = 1+r,
$$
where
$$
\CC_{\widehat{Bel}}\left(\frac{S_1}{S_0}\right) =
\sum\limits_{B \in \U} \left(\frac{S_1}{S_0}\right)^\l(B)\widehat{\mu}(B)
=
\sum_{i=1}^n m_i \left(\sum_{\{i\} \subseteq B \subseteq \{1,\ldots,i\}} \widehat{\mu}(B)\right).
$$

Hence, the above problem \eqref{eq:op} is equivalent to the following optimization problem with linear constraints, whose unknowns are the values of $\widehat{\mu}$ on $\U$:
\begin{equation}
\label{eq:lp}
\begin{array}{c}
\minimize d(\widehat{Bel},\underline{Q})\\[3ex]
\mbox{subject to:}\\[1ex]
\left\{
\begin{array}{ll}
\sum\limits_{\emptyset \neq B \subseteq A} \widehat{\mu}(B) \ge \underline{Q}(A), & \mbox{for every $A \in \U$},\\[1ex]
\sum\limits_{i=1}^n m_i \left(\sum\limits_{\{i\} \subseteq B \subseteq \{1,\ldots,i\}} \widehat{\mu}(B)\right) = 1+r,\\[3ex]
\sum\limits_{B \in \U} \widehat{\mu}(B) = 1,\\[2ex]
\widehat{\mu}(B) \ge 0, & \mbox{for every $B \in \U$}.
\end{array}
\right.
\end{array}
\end{equation}

It is easily seen that every $Q \in \cl(\Q)$ gives rise to a M\"obius inverse that satisfies all the constraints in problem \eqref{eq:lp}. Hence, the feasible region of problem \eqref{eq:lp} is a non-empty convex compact subset of $\RR^{(2^n -1)}$, endowed with the product topology.

We notice that, as already pointed out in \cite{mmv,mmv2}, if we consider the distance $d_1$ and take into account that $\widehat{Bel} \ge \underline{Q}$, we have that
\begin{eqnarray}
\label{eq:d1-simp}
d_1(\widehat{Bel},\underline{Q}) &=&
\sum\limits_{A \in \U}\left[\left(\sum\limits_{\emptyset \neq B \subseteq A} \widehat{\mu}(B)\right) - \underline{Q}(A)\right]\\\nonumber
&=&
\sum\limits_{A \in \U} 2^{|\Omega \setminus A|}\widehat{\mu}(A) - \sum\limits_{A \in \U}\underline{Q}(A),
\end{eqnarray}
where $\sum\limits_{A \in \U}\underline{Q}(A)$ is a constant, since $\underline{Q}$ is given. Therefore, problem \eqref{eq:lp} reduces to a linear programming problem.

The following example shows the computation of an equivalent inner approximating martingale belief function, relying on the distance $d_1$.

\begin{example}
Consider $\Omega = \{1,2,3,4\}$, $m_1 = 4$, $m_2 = 2$, $m_3 = \frac{1}{2}$, $m_4 = \frac{1}{4}$, $1+r = 1$, and $\Q$ and $\underline{Q}$ of Example~\ref{ex:lp}.

An inner approximating martingale belief function $\widehat{Bel}$ minimizing the $d_1$ distance is reported below

\begin{center}
\begin{tabular}{c|cccccccccccccccc}
$\F$ & $\emptyset$ & $1$ & $2$ & $3$ & $4$ & $12$ & $13$ & $14$ & $23$ & $24$ & $34$ & $123$ & $124$ & $134$ & $234$ & $\Omega$\\
\hline
\\[-1.5ex]
$\underline{\QQ}$
& $0$ & $0$ & $0$ & $0$ & $0$
& $\frac{15}{105}$ & $0$ & $0$ & $0$ & $0$ & $\frac{60}{105}$ &
$\frac{21}{105}$ & $\frac{15}{105}$ & $\frac{60}{105}$ & $\frac{84}{105}$ & $1$\\[1ex]
$\mu$
& $0$ & $0$ & $0$ & $0$ & $0$
& $\frac{15}{105}$ & $0$ & $0$ & $0$ & $0$ & $\frac{60}{105}$ &
$\frac{6}{105}$ & $0$ & $0$ & $\frac{24}{105}$ & $0$\\[1ex]
\hline
\\[-1.5ex]
$\widehat{\mu}$
& $0$ & $\frac{21}{105}$ & $0$ & $0$ & $0$
& $0$ & $0$ & $0$ & $0$ & $0$ & $\frac{60}{105}$ &
$0$ & $0$ & $0$ & $\frac{24}{105}$ & $0$\\[1ex]
$\widehat{Bel}$
& $0$ & $\frac{21}{105}$ & $0$ & $0$ & $0$
& $\frac{21}{105}$ & $\frac{21}{105}$ & $\frac{21}{105}$ & $0$ & $0$ & $\frac{60}{105}$ &
$\frac{21}{105}$ & $\frac{21}{105}$ & $\frac{81}{105}$ & $\frac{84}{105}$ & $1$
\end{tabular}
\end{center}
for which we have that $d_1(\widehat{Bel},\underline{Q}) = \frac{96}{105}$.

Define $Q_0 = \frac{1}{|\ext(\cl(\Q))|}\sum_{Q \in \ext(\cl(\Q))} Q$, whose values are reported below
\begin{center}
\begin{tabular}{c|cccccccccccccccc}
$\F$ & $\emptyset$ & $1$ & $2$ & $3$ & $4$ & $12$ & $13$ & $14$ & $23$ & $24$ & $34$ & $123$ & $124$ & $134$ & $234$ & $\Omega$\\
\hline
\\[-1.5ex]
$\QQ_{1,3}$
& $0$ & $\frac{15}{105}$ & $0$ & $\frac{90}{105}$ & $0$
& $\frac{15}{105}$ & $1$ & $\frac{15}{105}$ & $\frac{90}{105}$ & $0$ & $\frac{90}{105}$ &
$1$ & $\frac{15}{105}$ & $1$ & $\frac{90}{105}$ & $1$
\\[1ex]
$\QQ_{1,4}$
& $0$ & $\frac{21}{105}$ & $0$ & $0$ & $\frac{84}{105}$
& $\frac{21}{105}$ & $\frac{21}{105}$ & $1$ & $0$ & $\frac{84}{105}$ & $\frac{84}{105}$ &
$\frac{21}{105}$ & $1$ & $1$ & $\frac{84}{105}$ & $1$
\\[1ex]
$\QQ_{2,3}$
& $0$ & $0$ & $\frac{35}{105}$ & $\frac{70}{105}$ & $0$
& $\frac{35}{105}$ & $\frac{70}{105}$ & $0$ & $1$ & $\frac{35}{105}$ & $\frac{70}{105}$ &
$1$ & $\frac{35}{105}$ & $\frac{70}{105}$ & $1$ & $1$
\\[1ex]
$\QQ_{2,4}$
& $0$ & $0$ & $\frac{45}{105}$ & $0$ & $\frac{60}{105}$
& $\frac{45}{105}$ & $0$ & $\frac{60}{105}$ & $\frac{45}{105}$ & $1$ & $\frac{60}{105}$ &
$\frac{45}{105}$ & $1$ & $\frac{60}{105}$ & $1$ & $1$\\[1ex]
\hline
\\[-1.5ex]
$Q_0$
& $0$ & $\frac{36}{420}$ & $\frac{80}{420}$ & $\frac{160}{420}$ & $\frac{144}{420}$
& $\frac{116}{420}$ & $\frac{196}{420}$ & $\frac{180}{420}$ & $\frac{240}{420}$ & $\frac{224}{420}$ & $\frac{304}{420}$ &
$\frac{276}{420}$ & $\frac{260}{420}$ & $\frac{340}{420}$ & $\frac{384}{420}$ & $1$
\end{tabular}
\end{center}

Finally, for $\epsilon = \frac{1}{2}$, define $\widehat{Bel}_\epsilon = \frac{1}{2} Q_0 + \frac{1}{2} \widehat{Bel}$, whose values are reported below
\begin{center}
\begin{tabular}{c|cccccccccccccccc}
$\F$ & $\emptyset$ & $1$ & $2$ & $3$ & $4$ & $12$ & $13$ & $14$ & $23$ & $24$ & $34$ & $123$ & $124$ & $134$ & $234$ & $\Omega$\\
\hline
\\[-1.5ex]
$\widehat{Bel}_\epsilon$
& $0$ & $\frac{60}{420}$ & $\frac{40}{420}$ & $\frac{80}{420}$ & $\frac{72}{420}$
& $\frac{100}{420}$ & $\frac{140}{420}$ & $\frac{132}{420}$ & $\frac{120}{420}$ & $\frac{112}{420}$ & $\frac{272}{420}$ &
$\frac{180}{420}$ & $\frac{172}{420}$ & $\frac{332}{420}$ & $\frac{360}{420}$ & $1$
\end{tabular}
\end{center}
We have that $\widehat{Bel}_\epsilon$ is an equivalent inner approximating martingale belief function for $\underline{Q}$, furthermore, it is the lower envelope of the class of probability measures on $\F$
$$
\widehat{\Q}_\epsilon = \{Q \in \MM(\Omega,\F) \,:\, Q = (1-\epsilon)Q_0 + \epsilon Q', Q' \in \core(\widehat{Bel})\}.
$$
A direct computation shows that
$$
\CC_{\widehat{Bel}_\epsilon}\left(\frac{S_1}{S_0}\right) = \min_{Q \in \widehat{\Q}_\epsilon} \EE_Q\left(\frac{S_1}{S_0}\right) = 1+r.
$$
\hfill$\blacklozenge$
\end{example}

Despite using $d_1$ we get a linear programming problem, the main disadvantage is that the optimal solution is generally not unique, as shown in the following example.

\begin{example}
\label{ex:non-uniqued1}
Consider $\Omega = \{1,2,3,4\}$, $m_1 = 5$, $m_2 = 3$, $m_3 = 2$, $m_4 = \frac{1}{2}$ and $1+r = 4$.
According to Theorem~\ref{th:extp}, we have $I = \{1\}$, $J = \{2,3,4\}$ and $\ext(\cl(\Q)) = \{\QQ_{1,2}, \QQ_{1,3}, \QQ_{1,4}\}$ inducing $\underline{\QQ}$ and $\mu$ reported below
\begin{center}
\begin{tabular}{c|cccccccccccccccc}
$\F$ & $\emptyset$ & $1$ & $2$ & $3$ & $4$ & $12$ & $13$ & $14$ & $23$ & $24$ & $34$ & $123$ & $124$ & $134$ & $234$ & $\Omega$\\
\hline
\\[-1.5ex]
$\QQ_{1,2}$
& $0$ & $\frac{9}{18}$ & $\frac{9}{18}$ & $0$ & $0$ & $1$
& $\frac{9}{18}$ & $\frac{9}{18}$ & $\frac{9}{18}$ & $\frac{9}{18}$ & $0$ & $1$ &
$1$ & $\frac{9}{18}$ & $\frac{9}{18}$ & $1$
\\[1ex]
$\QQ_{1,3}$
& $0$ & $\frac{12}{18}$ & $0$ & $\frac{6}{18}$ & $0$
& $\frac{12}{18}$ & $1$ & $\frac{12}{18}$ & $\frac{6}{18}$ & $0$ & $\frac{6}{18}$ &
$1$ & $\frac{12}{18}$ & $1$ & $\frac{6}{18}$ & $1$
\\[1ex]
$\QQ_{1,4}$
& $0$ & $\frac{14}{18}$ & $0$ & $0$ & $\frac{4}{18}$
& $\frac{14}{18}$ & $\frac{14}{18}$ & $1$ & $0$ & $\frac{4}{18}$ & $\frac{4}{18}$ &
$\frac{14}{18}$ & $1$ & $1$ & $\frac{4}{18}$ & $1$
\\[1ex]
\hline
\\[-1.5ex]
$\underline{\QQ}$
& $0$ & $\frac{9}{18}$ & $0$ & $0$ & $0$
& $\frac{12}{18}$ & $\frac{9}{18}$ & $\frac{9}{18}$ & $0$ & $0$ & $0$ &
$\frac{14}{18}$ & $\frac{12}{18}$ & $\frac{9}{18}$ & $\frac{4}{18}$ & $1$\\[1ex]
$\mu$
& $0$ & $\frac{9}{18}$ & $0$ & $0$ & $0$
& $\frac{3}{18}$ & $0$ & $0$ & $0$ & $0$ & $0$ &
$\frac{2}{18}$ & $0$ & $0$ & $\frac{4}{18}$ & $0$
\end{tabular}
\end{center}

The following two belief functions have M\"obius inverse minimizing the distance $d_1$
\begin{center}
\begin{tabular}{c|cccccccccccccccc}
$\F$ & $\emptyset$ & $1$ & $2$ & $3$ & $4$ & $12$ & $13$ & $14$ & $23$ & $24$ & $34$ & $123$ & $124$ & $134$ & $234$ & $\Omega$\\
\hline
\\[-1.5ex]
$\widehat{\mu}_1$
& $0$ & $\frac{12}{18}$ & $0$ & $0$ & $0$ & $0$
& $0$ & $0$ & $\frac{4}{18}$ & $0$ & $0$ & $\frac{2}{18}$ &
$0$ & $0$ & $0$ & $0$
\\[1ex]
$\widehat{Bel}_1$
& $0$ & $\frac{12}{18}$ & $0$ & $0$ & $0$ & $\frac{12}{18}$
& $\frac{12}{18}$ & $\frac{12}{18}$ & $\frac{4}{18}$ & $0$ & $0$ & $1$ &
$\frac{12}{18}$ & $\frac{12}{18}$ & $\frac{4}{18}$ & $1$
\\[1ex]
\hline
\\[-1.5ex]
$\widehat{\mu}_2$
& $0$ & $\frac{11}{18}$ & $0$ & $0$ & $0$ & $\frac{3}{18}$
& $0$ & $0$ & $\frac{4}{18}$ & $0$ & $0$ & $0$ &
$0$ & $0$ & $0$ & $0$
\\[1ex]
$\widehat{Bel}_2$
& $0$ & $\frac{11}{18}$ & $0$ & $0$ & $0$ & $\frac{14}{18}$
& $\frac{11}{18}$ & $\frac{11}{18}$ & $\frac{4}{18}$ & $0$ & $0$ & $1$ &
$\frac{14}{18}$ & $\frac{11}{18}$ & $\frac{4}{18}$ & $1$
\end{tabular}
\end{center}
and it holds that $d_1(\widehat{Bel}_1,\underline{Q}) = d_1(\widehat{Bel}_2,\underline{Q}) = \frac{20}{18}$.
\hfill$\blacklozenge$
\end{example}

On the other hand, taking the distance $d_2$, then problem \eqref{eq:lp} admits a unique optimal solution as the objective function turns out to be strictly convex (see, e.g., \cite{mmv,mmv2}). In other terms, the choice of $d_2$ amounts in computing the orthogonal projection of $\underline{Q}$ onto the set of inner approximating martingale belief functions for $\underline{Q}$.

\begin{example}
Consider $\Omega$, $m_1$, $m_2$, $m_2$, $m_4$, and $1+r$ as in Example~\ref{ex:non-uniqued1}. In this case, the unique optimal solution $\widehat{Bel}$ minimizing $d_2$ has M\"obius inverse $\widehat{\mu}$ such that
$\widehat{\mu}(1) = 0.628655$,
$\widehat{\mu}(2) = 0.0087719$,
$\widehat{\mu}(12) = 0.149123$,
$\widehat{\mu}(23) = 0.18421$,
$\widehat{\mu}(234) = 0.0292399$,
and $0$ otherwise. In this case we have $d_2(\widehat{Bel},\underline{Q}) = 0.169591$.
\hfill$\blacklozenge$
\end{example}

Up to now, we have considered only the lower price assessment $\underline{\pi}(S_1) = S_0$. If we further impose to respect the upper price assessment $\overline{\pi}(S_1)= S_0$, then the notion of martingale belief function given in Definition~\ref{def:martbel} can be strengthened as follows.

\begin{definition}
A belief function $\widehat{Bel}$ on $\F$ is said:
\begin{itemize}
\item a {\bf strong martingale belief function} if
$$
(1+r)^{-1}\CC_{\widehat{Bel}}(S_1) = S_0
\quad \mbox{and} \quad
(1+r)^{-1}\CC_{\widehat{Bel}}(-S_1) = -S_0;
$$
\item an {\bf inner approximating strong martingale belief function} for $\underline{Q}$ if it is both an inner approximation for $\underline{Q}$ and a strong martingale belief function;
\item an {\bf equivalent inner approximating strong martingale belief function} for $\underline{Q}$ if it is an inner approximating strong martingale belief function for $\underline{Q}$ and $\widehat{Bel} \sim P$.
\end{itemize}
\end{definition}

Still referring to a distance $d$ defined on $\BB(\Omega,\F)$, an optimal inner approximating strong martingale belief function $\widehat{Bel}$ for $\underline{Q}$ can be found by solving the following optimization problem:
\begin{equation}
\label{eq:strongmart}
\begin{array}{c}
\minimize d(\widehat{Bel},\underline{Q})\\[1ex]
\mbox{subject to:}\\[1ex]
\left\{
\begin{array}{ll}
\widehat{Bel}(A) \ge \underline{Q}(A), & \mbox{for every $A \in \F$},\\[1ex]
(1+r)^{-1}\CC_{\widehat{Bel}}(S_1) = S_0,\\[1ex]
(1+r)^{-1}\CC_{\widehat{Bel}}(-S_1) = -S_0,\\[1ex]
\widehat{Bel} \in \BB(\Omega,\F).
\end{array}
\right.
\end{array}
\end{equation}
Also in this case, problem \eqref{eq:strongmart} can be reformulated as follows
\begin{equation}
\label{eq:lp-strong}
\begin{array}{c}
\minimize d(\widehat{Bel},\underline{Q})\\[3ex]
\mbox{subject to:}\\[1ex]
\left\{
\begin{array}{ll}
\sum\limits_{\emptyset \neq B \subseteq A} \widehat{\mu}(B) \ge \underline{Q}(A), & \mbox{for every $A \in \U$},\\[1ex]
\sum\limits_{i=1}^n m_i \left(\sum\limits_{\{i\} \subseteq B \subseteq \{1,\ldots,i\}} \widehat{\mu}(B)\right) = 1+r,\\[3ex]
\sum\limits_{i=1}^n m_i \left(\sum\limits_{\{i\} \subseteq B \subseteq \{i,\ldots,n\}} \widehat{\mu}(B)\right) = 1+r,\\[3ex]
\sum\limits_{B \in \U} \widehat{\mu}(B) = 1,\\[2ex]
\widehat{\mu}(B) \ge 0, & \mbox{for every $B \in \U$}.
\end{array}
\right.
\end{array}
\end{equation}

The following theorem states that any inner approximating strong martingale belief function $\widehat{Bel}$ for $\underline{Q}$ is actually a probability measure belonging to $\cl(\Q)$.

\begin{theorem}
\label{th:inner-strong}
For every distance $d$ defined on $\BB(\Omega,\F)$, the set of feasible solutions of problem \eqref{eq:strongmart} is $\cl(\Q)$.
Further, if $d = d_1$ then the set of optimal solutions of problem \eqref{eq:strongmart} coincides with $\cl(\Q)$, while if $d = d_2$ then there is a unique optimal solution.
\end{theorem}
\begin{proof}
First notice that inner approximating strong martingale belief functions for $\underline{Q}$, that is feasible solutions of problem \eqref{eq:strongmart}, are in one-to-one correspondence with feasible solutions of problem \eqref{eq:lp-strong}.
Define ${\cal V} = \U \setminus\{\{i\} \,:\, i \in \Omega\}$. Subtracting memberwise the second equation to the third equation of problem \eqref{eq:lp-strong} we get
$$
\sum_{B \in {\cal V}} \left(\max_{i \in B} m_i - \min_{i \in B} m_i \right) \widehat{\mu}(B) = 0.
$$
Hence, since for every $B \in {\cal V}$, $\left(\max\limits_{i \in B} m_i - \min\limits_{i \in B} m_i\right) > 0$ and $\widehat{\mu}(B) \ge 0$, any feasible solution of problem \eqref{eq:lp-strong} is such that $\widehat{\mu}(B) = 0$, for every $B \in {\cal V}$. In turn, this implies that any feasible solution of problem \eqref{eq:lp-strong} is the M\"obius inverse of a probability measure which is a feasible solution of problem \eqref{eq:strongmart}. Thus, if $\widehat{Bel}$ is a feasible solution of problem  \eqref{eq:strongmart} we have $\CC_{\widehat{Bel}}\left(\frac{S_1}{S_0}\right) = \EE_{\widehat{Bel}}\left(\frac{S_1}{S_0}\right) = 1+r$, implying that $\widehat{Bel} \in \cl(\Q)$. Vice versa, every element of $\cl(\Q)$ is easily seen to be a feasible solution of problem \eqref{eq:strongmart}.

If $d = d_1$, since every feasible solution $\widehat{Bel}$ of problem \eqref{eq:strongmart} is a probability measure with M\"obius inverse $\widehat{\mu}$, by \eqref{eq:d1-simp} we get that
$$
d_1(\widehat{Bel},\underline{Q}) =
\sum\limits_{\{i\} \in \U} 2^{|\Omega \setminus \{i\}|}\widehat{\mu}(\{i\}) - \sum\limits_{A \in \U}\underline{Q}(A)
= 2^{n-1} - \sum\limits_{A \in \U}\underline{Q}(A),
$$
that does not depend on $\widehat{Bel}$. Hence, all the elements of $\cl(\Q)$ are optimal according to $d_1$.

If $d = d_2$, then the uniqueness of the optimal solution immediately follows since the objective function of \eqref{eq:strongmart} is strictly convex.
\hfill$\square$
\end{proof}

Hence, using $d_1$ any element of $\cl(\Q)$ turns out to be optimal, while using $d_2$ we get the orthogonal projection of $\underline{Q}$ onto the set of inner approximating strong martingale belief functions for $\underline{Q}$, which is by Theorem~\ref{th:inner-strong} the set $\cl(\Q)$.

\begin{example}
Consider $\Omega$, $m_1$, $m_2$, $m_3$, $m_4$, and $1+r$ as in Example~\ref{ex:non-uniqued1}. Using $d_1$, the set of optimal inner approximating strong martingale belief functions for $\underline{Q}$ is $\cl(\Q)$ and for every $\widehat{Bel} \in \cl(\Q)$
we have $d_1(\widehat{Bel},\underline{Q}) = \frac{8}{3}$.

On the other hand, if we use the distance $d_2$ we have a unique optimal solution which is the following inner approximating strong martingale belief function (probability measure)
\begin{center}
\begin{tabular}{c|cccc}
$\Omega$ & $1$ & $2$ & $3$ & $4$ \\
\hline
\\[-1.5ex]
$\widehat{Bel}$ &
$0.638889$ &
$0.188596$ &
$0.102339$ &
$0.0701755$
\end{tabular}
\end{center}
for which we have $d_2(\widehat{Bel},\underline{Q}) = 0.572124$.
\hfill$\blacklozenge$
\end{example}

Let us stress that, for a fixed $Q_0 \in \Q$, if $\widehat{Bel}$ is an inner approximating strong martingale belief function  (probability measure), then $\core(\widehat{Bel}) = \{\widehat{Bel}\}$. In this case $\widehat{\Q}_\epsilon$ reduces to the singleton
$\widehat{\Q}_\epsilon = \{(1-\epsilon)Q_0 + \epsilon \widehat{Bel}\}$, thus the lower envelope $\underline{Q}_\epsilon$ is an equivalent martingale measure, that is $\underline{Q}_\epsilon = (1-\epsilon)Q_0 + \epsilon \widehat{Bel} \in \Q$.

The following proposition states that, both for $d_1$ and $d_2$, an optimal inner approximating martingale (strong martingale) belief function does not dominate any other inner approximating martingale (strong martingale) belief function. This is inline with results proved in \cite{mmv,mmv2}, where the problem of finding an outer approximating belief function for a lower probability is studied.

\begin{proposition}
Let $d = d_1$ or $d = d_2$. If $\widehat{Bel}$ is an optimal solution of problem \eqref{eq:op} (problem \eqref{eq:strongmart}) then there is no feasible solution $\widehat{Bel}'$ of problem \eqref{eq:op} (problem \eqref{eq:strongmart}) such that $\widehat{Bel}' \neq \widehat{Bel}$ and $\u{Q} \le \widehat{Bel}' \le \widehat{Bel}$.
\end{proposition}
\begin{proof}
The proof can be obtained by a straightforward adaptation of Lemma~14 in \cite{mmv}.
\hfill$\square$
\end{proof}

\section{Conclusions}
\label{sec:conclusions}
In this paper we characterize the lower envelope of the set of equivalent martingale measures arising in a one-period $n$-nomial market model, showing that it is a belief function. This suggests to use such lower envelope to derive a lower pricing rule accommodating frictions in the market, in the form of bid-ask spreads.

For that we formulate a general one-period pricing problem and prove a version of the first fundamental theorem of asset pricing in the context of belief functions. The theorem relies on a generalized definition of arbitrage assuming partially resolving uncertainty, according to  Jaffray.

Finally, we cope with the derivation of a generalized arbitrage-free lower pricing rule stemming from the ``risk-neutral'' belief function $\underline{Q}$ arising in the one-period $n$-nomial market model.
This amounts in choosing an equivalent martingale measure and in producing an $\epsilon$-contamination relying on a suitable inner approximation of $\underline{Q}$.

As a topic of future research, we aim at extending the introduced notion of arbitrage to the multi-period case. For this to be possible, the issue of dynamic consistency needs to be taken into account \cite{klr-dynamic,asano}.

\begin{acknowledgements}
The last two authors are members of the GNAMPA-INdAM research group. The second author was supported by Università degli Studi di Perugia, Fondo Ricerca di Base 2019, project ``Modelli per le decisioni economiche e finanziarie in condizioni di ambiguità ed imprecisione''.
\end{acknowledgements}



%

\bibliographystyle{spmpsci}
\bibliography{biblio}

%

\end{document}